\RequirePackage{fix-cm}
\documentclass[smallextended]{svjour3}       

\usepackage{amsmath}
\usepackage{amsfonts}
\usepackage{amssymb}
\usepackage{mathrsfs} 
\usepackage{bm}       
\usepackage{booktabs} 

\usepackage{graphicx}
\usepackage{subcaption}  

\usepackage{algorithm}
\usepackage{algorithmic}

\usepackage{color}

\usepackage{cite}
\usepackage{xr-hyper}
\usepackage{hyperref}
\newcommand{\V}[1]{\boldsymbol{#1}} 
\newcommand{\M}[1]{\boldsymbol{#1}} 
\newcommand{\Set}[1]{\mathbb{#1}}

\newcommand{\norm}[1]{\left\Vert #1\right\Vert }

\newcommand{\av}[1]{\left\langle #1\right\rangle }

\smartqed  
\usepackage{graphicx}
\begin{document}

\title{A spectral-subspace-augmented POD-Galerkin method for parametrized PDEs with limited snapshot data
}


\author{Tianhao Hu         \and
        Zecheng Gan 
}


\institute{Hu, T. \at
              Hong Kong University of Science and Technology (Guangzhou) \\
              \email{thu176@connect.hkust-gz.edu.cn}           
           \and
           Gan, Z. \at
              Hong Kong University of Science and Technology (Guangzhou) \\
              \email{zechenggan@hkust-gz.edu.cn} 
}

\date{Received: date / Accepted: date}

\maketitle

\begin{abstract}
Parametrized partial differential equations (PDEs) arise in many-query simulation, optimization, control, and uncertainty quantification, where repeated full-order solves restrict the number of high-fidelity snapshots that can be generated.
This limitation is particularly pronounced in computational energy science, where multiscale models of porous-media flow, transport, and energy materials often make large snapshot datasets impractical.
Proper orthogonal decomposition (POD) constructs compact reduced bases from solution snapshots, but it may exhibit	 limited out-of-sample predictive capability when the snapshots insufficiently sample the solution manifold.
To address this limitation, we propose a spectral-subspace-augmented POD-Galerkin method (SS-POD) tailored to limited-data regimes.
SS-POD starts from a problem-adapted spectral approximation space, partitions it into orthogonal subspaces, and performs POD locally on the projected snapshot matrices.
An energy-balancing rule determines the spectral partition so that the resulting local POD problems are assigned comparable amounts of snapshot energy.
For nonlinear parametrized PDEs, SS-POD is coupled with the discrete empirical interpolation method (DEIM).
Numerical experiments show that SS-POD improves out-of-sample accuracy over standard POD-Galerkin while retaining compact reduced bases in limited-snapshot regimes.
In particular, for a Laplace--Beltrami problem on the unit sphere with only 5 snapshots, SS-POD achieves a relative error of $3.9\times 10^{-8}$ using 91 basis functions, whereas the standard POD error saturates at $7.8\times 10^{-4}$ and the spectral-Galerkin method requires 256 basis functions for comparable accuracy.
These results indicate that SS-POD provides an effective strategy for high-fidelity reduced-order modeling from limited snapshot data.

\end{abstract}

\section{Introduction}\label{sec:intro2}
Parametrized partial differential equations (PDEs) arise in optimization, control, inverse problems, and uncertainty quantification, where the same governing model must be solved for many parameter values~\cite{hesthaven2016certified}. High-fidelity finite element, finite difference, and spectral discretizations can make such many-query studies computationally expensive~\cite{grepl2007certified,fox1971approximate,oliveira2007reduced}. Similar bottlenecks arise in computational energy science, including subsurface flow, transport, and data-driven energy-materials modeling, where simulations and high-fidelity data generation often span multiple scales~\cite{wang2026multifidelityjmi}. Reduced-order models (ROMs) reduce this cost by separating an offline basis-construction stage from a low-dimensional online solve. This strategy is effective when the solution manifold admits a compact approximation and the offline data sufficiently identify the relevant solution components~\cite{hesthaven2016certified,pinkus2012n,binev2011convergence}. When only a few high-fidelity snapshots are available, however, the reduced space may fail to capture components that are essential for out-of-sample prediction.

Reduced basis methods (RBMs) and proper orthogonal decomposition (POD) are two widely used approaches for constructing reduced spaces. RBMs select parameter samples adaptively, often through greedy algorithms guided by a posteriori error estimators~\cite{veroy2003posteriori,prud2002reliable,patera2007reduced,rozza2008reduced,quarteroni2015reduced,buffa2012priori}. In contrast, POD starts from a snapshot ensemble and computes an orthonormal basis that is optimal, in the sense of minimizing the mean-square projection error over the given snapshots. Introduced in the analysis of turbulent flows~\cite{lumley1967structure,lumey2012stochastic}, POD is closely related to principal component analysis, empirical orthogonal functions, and the Karhunen--Loeve expansion~\cite{abdi2010principal,hannachi2007empirical,loeve1977elementary}. Its strong compression capability has made POD a standard basis-construction tool in computational fluid dynamics, porous-media flow, and related applications~\cite{taira2017modal,gunzburger2012finite,ito1998reduced,ito2001reduced,ito2001reduced_2,veroy2005certified,deparis2008reduced,wang2016podgalerkinporous}.

POD depends on the empirical distribution of the available snapshots. When the snapshots are sparse or unevenly distributed, the resulting basis may accurately reproduce the sampled states but fail to represent unsampled regions of the solution manifold. Adding more POD modes within the same snapshot span cannot remedy this deficiency if that span does not contain the solution components needed for prediction. This limitation contrasts with greedy RBMs, which use an error indicator to target parameter values that are poorly represented by the current reduced space~\cite{haasdonk2017reduced}. 
Consequently, when snapshot data are insufficient, POD may incur substantial out-of-sample errors, especially for problems involving multiple scales, localized structures, high-frequency components, or strong parameter sensitivity.

Several POD variants incorporate additional structure into the modal decomposition. Spectral POD (SPOD) extracts frequency-resolved coherent structures in statistically stationary flows~\cite{sieber2016spectral,towne2018spectral}, with applications to jet dynamics~\cite{gudmundsson2011instability,sinha2014wavepacket,schmidt2017wavepackets}. Multiscale POD (mPOD) combines POD with multiresolution analysis by decomposing snapshots into prescribed frequency bands before computing local POD modes~\cite{mendez2019multi,mendez2020multiscale,mendez2018multiscale,mendez2022generalized}. These methods exploit spectral or multiscale information primarily to improve modal analysis and flow decomposition. In contrast, we use spectral subspaces to enrich POD-Galerkin reduced spaces for predictive reduced-order modeling of parametrized PDEs in limited-data regimes.

Machine-learning approaches provide another route for parametrized PDEs. Physics-informed neural networks (PINN) and Deep Ritz methods approximate PDE solutions from residual or variational formulations~\cite{raissi2019physics,yu2018deep}, while Deep BSDE methods target high-dimensional parabolic problems~\cite{han2017deep,han2018solving,weinan2021algorithms}. Operator-learning approaches, such as DeepONet and Fourier neural operators, learn solution maps directly from data~\cite{lu2021learning,li2020fourier}. Hybrid methods further integrate traditional numerical methods with data-driven modeling techniques, including random-feature PDE solvers~\cite{chen2022bridging}, non-intrusive ROMs for flow problems~\cite{xiao2015nonintrusive}, nonlinear reduced-basis methods with online neural adaptation~\cite{li2025nemytskii}, POD-enhanced deep ROMs~\cite{fresca2022pod}, neural ROM identification~\cite{wang2018model}, and PINN/RBM integrations~\cite{chen2024gpt,chen2024tgpt}. Although these approaches provide powerful alternatives, they often involve substantial training data, nonlinear optimization, or online adaptation. We instead focus on a complementary ROM setting: improving POD-Galerkin reduced bases in limited-snapshot regimes while retaining the transparency and projection structure of classical ROMs.

We propose a spectral-subspace-augmented POD-Galerkin method (SS-POD) for parametrized PDEs with limited snapshot data. SS-POD starts from a problem-adapted spectral approximation space and decomposes it into mutually orthogonal spectral subspaces. The snapshots are then projected onto these subspaces, POD bases are computed locally, and the resulting local bases are assembled into an augmented reduced basis. This construction mitigates the tendency of dominant high-energy snapshot components to obscure lower-energy components that may be relevant for out-of-sample prediction, while preserving the projection structure and interpretability of classical POD-Galerkin ROMs. 
An energy-balancing rule determines the spectral partition by assigning comparable amounts of snapshot energy to the local POD problems. We also extend the construction to nonlinear parametrized PDEs by coupling SS-POD with the discrete empirical interpolation method (DEIM)~\cite{barrault2004empirical,chaturantabut2009discrete,chaturantabut2010nonlinear,grepl2007efficient}. 
The main contributions of this work are: (i) a spectral-subspace-augmented reduced basis for data-scarce POD-Galerkin modeling; (ii) an energy-balancing rule for selecting the spectral partition; and (iii) a DEIM extension for nonlinear terms. We validate the proposed method through numerical comparisons with POD-Galerkin and spectral-Galerkin baselines on five PDE benchmarks.

The rest of the paper is organized as follows. Section~\ref{sec:POD} reviews POD-Galerkin ROMs and the offline-online decomposition. Section~\ref{sec:SS-POD} presents the SS-POD construction, the energy-balancing rule, the nonlinear extension, and the reconstruction estimate. Section~\ref{sec:numerical_result} presents numerical benchmark results. Section~\ref{sec:conclusion} concludes the paper and discusses limitations and possible extensions.

\section{POD-Galerkin Method Revisited}
\label{sec:POD}
\label{sec:Galerkin_in_ROM}

In this section, we briefly review the POD-Galerkin method. 
For simplicity, consider a linear elliptic problem with homogeneous Dirichlet boundary conditions parameterized by $\mu$,
\begin{equation}
\label{eqn:general_gov_eqn}
    \mathcal{L}(\mu) u(\V{x}; \mu) = f(\V{x}; \mu), \quad \V{x} \in \Omega,
\end{equation}
where $\mu \in \Set{P}$ and $\Set{P}$ denotes the parameter space. After spatial discretization, we write $\Set{V}$ for the finite-dimensional trial space. The space $\Set{V}$ is equipped with the inner product $\av{\cdot,\cdot}_{\Set{V}}$ and the induced norm $\norm{\V v}_{\Set{V}}=\sqrt{\av{\V v,\V v}_{\Set{V}}}$. For the homogeneous Dirichlet case, $\Set{V}$ is the discrete analogue of $H_0^1(\Omega)$, so the boundary condition is encoded in the approximation space.

The full-order variational problem associated with Eq.~\eqref{eqn:general_gov_eqn} is written as: find $\V u(\mu)\in\Set{V}$ such that
\begin{equation}\label{eqn:variational_problem}
    A(\V u(\mu), \V v; \mu) = F(\V v; \mu), \quad \forall \V v \in \Set{V}_t,
\end{equation}
where $A(\cdot,\cdot;\mu)$ and $F(\cdot;\mu)$ denote the discrete bilinear and linear forms, $\V u(\mu)$ denotes the discrete solution in $\Set{V}$, and $\Set{V}_t$ the test space. In standard Galerkin projection, $\Set{V}_t = \Set{V}$.

Many-query applications require this full-order problem to be solved for many values of $\mu$. Solving it directly each time is usually too expensive. POD-Galerkin reduction therefore separates the computation into an offline stage, where a reduced space is learned, and an online stage, where only the reduced coefficients are solved.

In the offline stage, snapshots $\M U=[\V u_1,\dots,\V u_{N_{\mathrm{s}}}]$ are collected from full-order solutions, and POD is used to extract a reduced basis $\M\Phi=\M\Phi_{\mathrm{POD}}$ from the snapshot matrix. In the online formulation below, we write $\M\Phi=[\V\phi_1,\dots,\V\phi_{N_{\mathrm{p}}}]$ ($N_{\mathrm{p}} < N_{\mathrm{g}}$) for a generic reduced basis, because the same formulation will later be used for SS-POD. 
The reduced approximation then seeks $\V u(\mu)$ in the span of these $N_{\mathrm{p}}$ modes,
\begin{equation}\label{eqn:galerkin_approx}
    \V u(\mu) \approx \hat{\V u}(\mu)
    =
    \sum_{i=1}^{N_{\mathrm{p}}} a_i(\mu)\V\phi_i
    =
    \M\Phi\V a(\mu),
\end{equation}
where $\V a(\mu)=[a_1(\mu),\dots,a_{N_{\mathrm{p}}}(\mu)]^\mathrm{T}$ contains the reduced coefficients. 

Let $\Set{V}_t$ be the reduced test space spanned by $\tilde{\M\Phi}=[\tilde{\V\phi}_1,\dots,\tilde{\V\phi}_{N_{\mathrm{p}}}]$. 
Substituting Eq.~\eqref{eqn:galerkin_approx} into Eq.~\eqref{eqn:variational_problem} and testing against $\tilde{\V\phi}_j\in\Set{V}_t$ gives the reduced system
\begin{equation}
    \sum_{i=1}^{N_{\mathrm{p}}}
    A(\V\phi_i,\tilde{\V\phi}_j;\mu)a_i(\mu)
    =
    F(\tilde{\V\phi}_j;\mu),
    \quad j=1,\dots,N_{\mathrm{p}}\;.
\end{equation}
Note that for SS-POD, the test basis not necessary coincides with the trial basis $\M\Phi$; problem-specific choices of the test space will be discussed later and presented in the numerical examples.
The projection step reduces the online solve to an $N_{\mathrm{p}}\times N_{\mathrm{p}}$ linear system. 

This model reduction strategy is efficient if two requirements are met. First, the solution manifold $\Set{M}=\{\V u(\mu)\mid \mu\in\Set{P}\}$ must be well approximated by a low-dimensional space. This property is often associated with rapid decay of the Kolmogorov $N$-width~\cite{pinkus2012n,binev2011convergence}. Second, the reduced operators must be evaluated without incurring full-order cost.
A standard assumption used to satisfy the second requirement is an affine parameter decomposition,
\begin{equation}
\label{eqn:affine_decomposition_operator}
    \mathcal{L}(\mu) = \sum_{q=1}^{Q} B_q(\mu) \mathcal{L}_q,
\end{equation}
where $\mathcal{L}_q$ are parameter-independent operators and $B_q(\mu)$ are scalar functions of $\mu$. Under this affine decomposition, the reduced matrices associated with $\mathcal{L}_q$ can be precomputed offline and assembled online through the coefficients $B_q(\mu)$. For nonlinear or non-affine operators, this separation generally requires an additional hyper-reduction step. In this work, such terms are handled by combining SS-POD with the discrete empirical interpolation method (DEIM), which will be discussed in Sec.~\ref{sec:DEIM}.

\section{The spectral-subspace-augmented POD-Galerkin (SS-POD) method}\label{sec:SS-POD}
\subsection{Construction of the SS-POD Methodology}
\subsubsection{Hierarchy of Projection Operators and Subspaces}\label{sec:SSPOD_project_operator}
POD extracts coherent structures from solution snapshots and provides the optimal low-rank approximation of the snapshot matrix in the chosen norm. Its success in ROMs is closely related to the Kolmogorov $N$-width of the solution manifold~\cite{pinkus2012n,hesthaven2016certified,volkwein2013proper}. This connection, however, should be interpreted with care: POD is optimal for the available snapshot ensemble, not automatically for the full parametric manifold. When the training set is sparse or poorly distributed, the empirical snapshot span may miss directions that are important for out-of-sample parameters.

Spectral methods provide an a priori approximation space independent of the sampled parameters. Such spaces tolerate sparse training data better than empirical bases, but they may require more modes than a POD basis fitted to the target solution family. SS-POD combines these two ingredients by projecting snapshots onto orthogonal spectral subspaces and then performing POD locally. The spectral prior prevents dominant empirical directions from determining the whole reduced space, while the local POD steps keep data adaptation within each spectral window. Fig.~\ref{fig:workflow_SSPOD} summarizes the construction.

\begin{figure}[ht]
    \centering
    \includegraphics[width=1\textwidth]{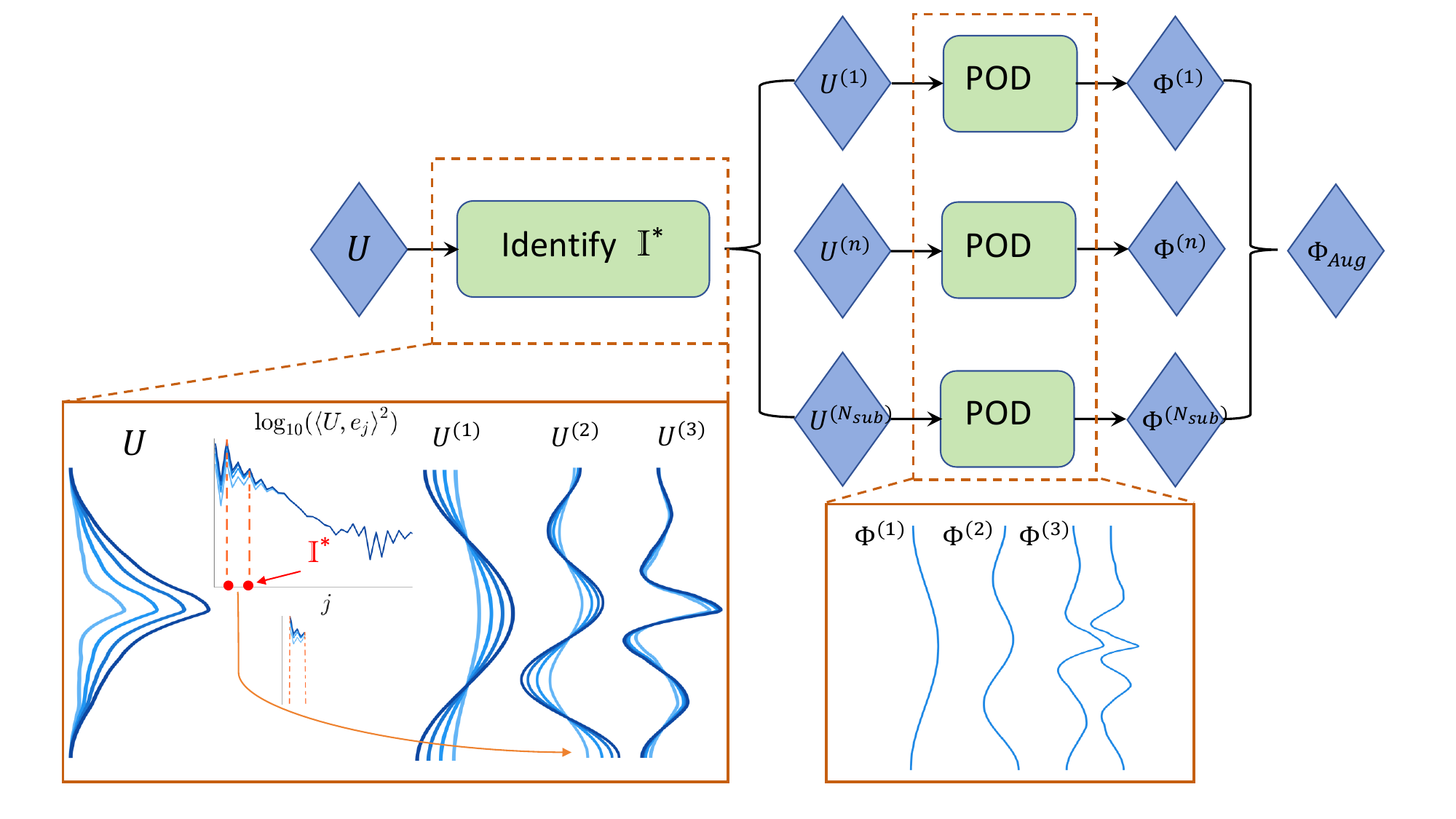}
    \caption{Workflow of the SS-POD methodology.}
    \label{fig:workflow_SSPOD}
\end{figure}

Let $\M\Phi_{\mathrm{spec}}=[\V e_1,\dots,\V e_{N_{\mathrm{max}}}]$ be a collection of $N_{\mathrm{max}}$ normalized spectral basis functions. We assume that $N_{\mathrm{max}}$ is large enough for the truncated spectral space $\Set E=\mathrm{span}\{\V e_1,\dots,\V e_{N_{\mathrm{max}}}\}$ to approximate the relevant solution states to the desired accuracy. Depending on the geometry and boundary conditions, the basis $\{\V e_j\}$ may be chosen from Fourier modes, Chebyshev polynomials, spherical harmonics, or other problem-adapted spectral functions. In the discrete setting, the modes are normalized so that $\av{\V e_i,\V e_j}_{\Set V}=\delta_{ij}$. The space $\Set E$ is decomposed into orthogonal subspaces,
\begin{equation}
    \Set{E} = \mathrm{span}\{ \V e_1, \dots, \V e_{N_{\mathrm{max}}} \} = \Set{E}_1 \oplus \Set{E}_2 \oplus \dots \oplus \Set{E}_{N_{\mathrm{sub}}},
\end{equation}
where $N_{\mathrm{sub}}$ is the number of spectral windows. Each subspace is spanned by a consecutive subset of spectral modes,
\begin{equation}
    \Set{E}_n = \mathrm{span} \{ \V e_{k_n+1}, \dots, \V e_{k_{n+1}} \},
\end{equation}
where the decomposition indices satisfy $k_1=0$, $k_{N_{\mathrm{sub}}+1}=N_{\mathrm{max}}$, and $k_i<k_j$ for $i<j$.
SS-POD selects an energy-balanced decomposition index set from the spectral coefficients of the available snapshots:
\begin{equation}
    \Set{I}^* = \{ k_1^*, \dots, k_{N_{\mathrm{sub}}+1}^* \},
\end{equation}
The procedure for selecting $\Set{I}^*$ is detailed in Sect.~\ref{subsec:IROSG}.

Given snapshots $\V u_i$ ($i=1,\dots,N_{\mathrm s}$) and the partition $\Set I^*$, the data are projected onto each subspace $\Set E_n$ by the $\Set V$-orthogonal projection $\M\Pi_n:\Set R^{N_{\mathrm g}}\to\Set R^{N_{\mathrm g}}$:
\begin{equation}
    \V u_i^{(n)} = \M \Pi_n \V u_i = \sum_{j= k_n +1}^{k_{n+1}} \av{\V u_i, \V e_{j}}_{\Set{V}} \V e_{j}, \quad i=1, \dots, N_{\mathrm{s}}.
\end{equation}
POD is then applied separately to the projected snapshot matrices $\M U^{(n)}=[\V u_1^{(n)},\dots,\V u_{N_{\mathrm s}}^{(n)}]$, yielding local bases $\M\Phi^{(n)}=[\V\phi^{(n)}_1,\dots,\V\phi^{(n)}_{N_{\mathrm p}^{(n)}}]$. The SS-POD basis is obtained by concatenating the local bases:
\begin{equation}
     \M \Phi_{\mathrm{Aug}} = [\M \Phi^{(1)}, \dots, \M \Phi^{(N_{\mathrm{sub}})}].
\end{equation}
The SS-POD basis set construction interpolates between two limiting cases:
\begin{equation} \label{ROSG_property}
     \M \Phi_{\mathrm{Aug}} = 
    \begin{cases}
         \M \Phi_{\mathrm{POD}}, & \text{if } N_{\mathrm{sub}}=1 \text{ and } N_{\mathrm{p}}^{(1)} = N_{\mathrm{p}}, \\
         \M \Phi_{\mathrm{spec}}, & \text{if } N_{\mathrm{sub}}=N_{\mathrm{max}} \text{ and } N^{(n)}_{\mathrm{p}} = 1, \, n=1, \dots, N_{\mathrm{sub}}.
    \end{cases} 
\end{equation}
The first case recovers standard POD. The second formally recovers the truncated spectral basis when every spectral window contains one active mode and the corresponding projected snapshot matrix has nonzero energy in that mode. Thus SS-POD can be viewed as a controlled interpolation between empirical POD and spectral Galerkin spaces. The parameters $N_{\mathrm{sub}}$ and $N_{\mathrm p}^{(n)}$ determine the tradeoff among accuracy, online dimension, and snapshot requirements. Algorithm~\ref{alg1:ROSG} summarizes the construction of $\M\Phi_{\mathrm{Aug}}$.

\begin{algorithm}[ht]
    \caption{Calculation of $\M \Phi_{\mathrm{Aug}}$}
    \label{alg1:ROSG}
    \begin{algorithmic}[1]
        \REQUIRE Snapshots $\M U = [\V u_1, \dots, \V u_{N_{\mathrm{s}}}]$, number of subspaces $N_{\mathrm{sub}}$, maximum number of spectral basis functions $N_{\mathrm{max}}$.
        \ENSURE Subspace-augmented basis $\M \Phi_{\mathrm{Aug}}$.

        \STATE Generate the appropriate spectral basis $\{ \V e_j \}_{j=1}^{N_{\mathrm{max}}}$.
        \STATE Determine the energy-balanced partition $\Set{I}^* = \{ k_1, \dots, k_{N_{\mathrm{sub}}+1} \}$ using Algorithm~\ref{alg4:I_ROSG} with inputs $\av{\V u_i, \V e_j}_{\Set{V}}$ and $N_{\mathrm{sub}}$.
        \STATE For each subspace $\Set{E}_n = \mathrm{span}\{ \V e_{k_n+1}, \dots, \V e_{k_{n+1}} \}$:
            \begin{itemize}
                \item Calculate the projections $\V u_{i}^{(n)} = \sum_{j= k_n +1}^{k_{n+1}} \av{\V u_{i}, \V e_{j}}_{\Set{V}} \V e_{j}$ for $i=1, \dots, N_{\mathrm{s}}$.
                \item Form the projected snapshot matrix $\M U^{(n)} = [\V u_{1}^{(n)}, \dots, \V u_{N_{\mathrm{s}}}^{(n)}]$.
            \end{itemize}
        \STATE Perform truncated SVD on each $\M U^{(n)}$ to obtain the basis $\M \Phi^{(n)} = [\V \phi^{(n)}_1, \dots, \V \phi^{(n)}_{N_{\mathrm{p}}^{(n)}}]$.
        \STATE Concatenate the bases to obtain:
            \[
                \M \Phi_{\mathrm{Aug}} = [\M \Phi^{(1)}, \dots, \M \Phi^{(N_{\mathrm{sub}})}].
            \]
    \end{algorithmic}
\end{algorithm}

\subsubsection{Energy-balanced Determination of the Decomposition Index Set $\Set{I}^*$}\label{subsec:IROSG}
SS-POD depends on the decomposition index set $\Set{I}^*$. At the two extremes it recovers standard POD and spectral-Galerkin methods; between those extremes, the energy-balanced partition controls how much snapshot information enters each spectral window.

The choice of $\Set I^*$ is based on the energy distribution of the snapshots in the spectral basis. Standard POD tends to prioritize directions with the largest empirical energy. For many PDE data sets, these directions are dominated by low-frequency or large-scale components, so lower-energy spectral ranges may receive too few basis vectors even when they are important for accuracy. SS-POD mitigates this imbalance by distributing the snapshot energy captured by the finite spectral prior approximately equally among the subspaces $\Set E_n$. Applying POD to each projected matrix $\M U^{(n)}$ then gives each spectral window its own local approximation budget.

Specifically, the energy captured by the finite spectral prior, denoted by $\mathscr{E}$, is evaluated by summing the squared spectral coefficients of the projected snapshots:
\begin{equation} \label{POD energy}
    \begin{aligned}
        \mathscr{E} & =\sum_{i=1}^{N_{\mathrm{s}}} \norm{ \sum_{n=1}^{N_{\mathrm{sub}}} \M \Pi_n  \V u_i }_{\Set{V}}^2 
         = \sum_{i=1}^{N_{\mathrm{s}}} \norm{ \sum_{j=1}^{N_{\mathrm{max}}} \av{\V u_i, \V e_j}_{\Set{V}} \V e_j }_{\Set{V}}^2 = \sum_{i=1}^{N_{\mathrm{s}}} \sum_{j=1}^{N_{\mathrm{max}}} \av{\V u_i, \V e_j}_{\Set{V}}^2.
    \end{aligned}
\end{equation}
The target energy for each subspace is defined as $\mathscr E_n=\mathscr E/N_{\mathrm{sub}}$. In the idealized case, the partition would satisfy
\begin{equation}\label{Eqn:energy_each_subspace}
        \sum_{i=1}^{N_{\mathrm{s}}} \sum_{j= k_n +1}^{k_{n+1}} \av{\V u_i, \V e_j}_{\Set{V}}^2 = \mathscr{E}_n.
\end{equation}
Discrete partition indices usually prevent exact equality. Algorithm~\ref{alg4:I_ROSG} constructs $\Set I^*$ by balancing energy over the ordered spectral coefficients.

\begin{figure}[ht]
    \centering
    \includegraphics[width=1\textwidth]{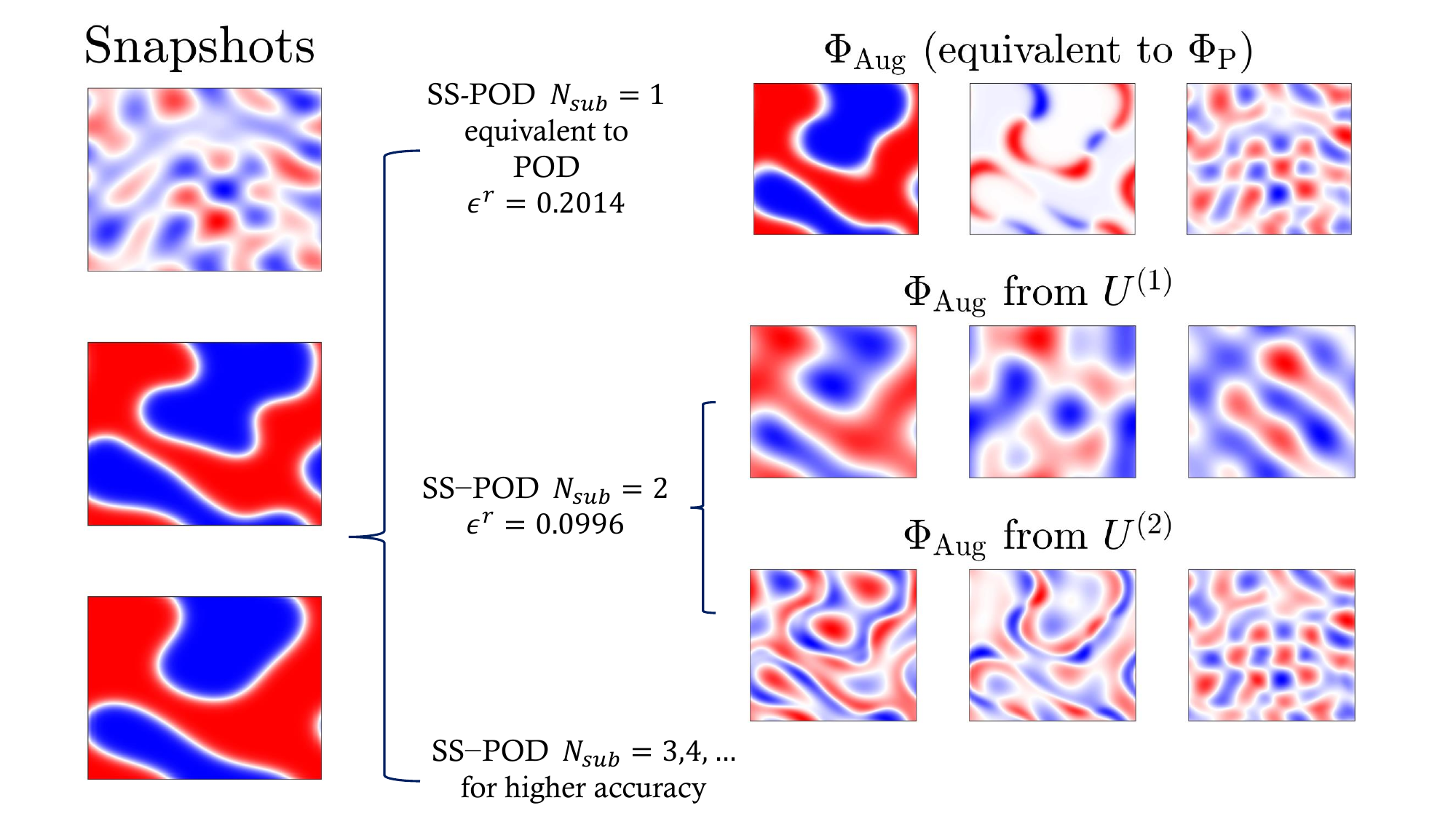}
    \caption{SS-POD basis obtained for the Allen-Cahn equation with three snapshots. The relative error $\epsilon^r$ is defined in Sect.~\ref{Tech_detail}.}
    \label{fig:ROSG_comparision}
\end{figure}

\begin{algorithm}[ht]
    \caption{Calculation of the index set $\Set{I}^*$}
    \label{alg4:I_ROSG}
    \begin{algorithmic}[1]
        \REQUIRE Spectral coefficients $\av{\V u_i, \V e_j}_{\Set{V}}$ for $i=1,\dots,N_{\mathrm{s}}$ and $j=1,\dots,N_{\mathrm{max}}$, and the number of subspaces $N_{\mathrm{sub}}$.
        \ENSURE Energy-balanced index set $\Set{I}^*$ and updated number of subspaces $N_{\mathrm{sub}}$.
            \STATE Calculate the projected spectral energy $\mathscr{E} = \sum_{i=1}^{N_{\mathrm{s}}} \sum_{j=1}^{N_{\mathrm{max}}} \av{\V u_i, \V e_j}_{\Set{V}}^2$.
            \STATE Initialization: $\Set{I}^* = \{0\}$, $\tilde{\mathscr{E}} = 0$, $\tilde{N}_{\mathrm{sub}} = 0$, and the target energy $\mathscr{E}_n = \mathscr{E} / N_{\mathrm{sub}}$, where $\tilde{N}_{\mathrm{sub}}$ counts completed subspaces.
        \FOR{$l = 1$ \TO $N_{\mathrm{max}}$}
            \STATE Accumulate energy of the current spectral mode: $\tilde{\mathscr{E}} = \tilde{\mathscr{E}} + \sum_{i=1}^{N_{\mathrm{s}}} \av{\V u_i, \V e_l}_{\Set{V}}^2$.
            \IF{$\tilde{\mathscr{E}} > \mathscr{E}_n$ and $\tilde{N}_{\mathrm{sub}} < N_{\mathrm{sub}}-1$}
                \STATE Update the index set: $\Set{I}^* = \{ \Set{I}^*, l \}$.
                \STATE Update subspace count: $\tilde{N}_{\mathrm{sub}} = \tilde{N}_{\mathrm{sub}} + 1$.
                \STATE Recalculate remaining average energy:
                \[
                \mathscr{E}_n = \frac{\mathscr{E} - \sum_{i=1}^{N_{\mathrm{s}}} \sum_{j=1}^{l} \av{\V u_i, \V e_j}_{\Set{V}}^2}{N_{\mathrm{sub}} - \tilde{N}_{\mathrm{sub}}}.
                \]
                \STATE Reset current subspace energy: $\tilde{\mathscr{E}} = 0$.
            \ENDIF
        \ENDFOR
        \STATE Set the final index: $\Set{I}^*(\text{end}) = N_{\mathrm{max}}$.
        \STATE Return the finalized $N_{\mathrm{sub}} = \tilde{N}_{\mathrm{sub}}+1$ and $\Set{I}^*$.
    \end{algorithmic}
\end{algorithm}

\subsubsection{Extension to Nonlinear PDEs via SS-POD/DEIM}\label{sec:DEIM}
Although SS-POD has been introduced for linear projection spaces, the same subspace construction can be coupled with the Discrete Empirical Interpolation Method (DEIM) to handle non-affine and nonlinear terms~\cite{barrault2004empirical,grepl2007efficient,chaturantabut2009discrete,chaturantabut2010nonlinear,chaturantabut2012state,li2020poddeimgas}. DEIM is a special case of the Empirical Interpolation Method (EIM) that selects representative interpolation indices from nonlinear basis functions, thereby reducing the cost of evaluating nonlinear reduced-order models.
Related hyper-reduction strategies, including reduced over-collocation, address the same full-grid nonlinear-evaluation bottleneck from a complementary sampling perspective~\cite{chen2021l1hyperreduction}.

In a projection-based ROM, the online stage should ideally be independent of the full-order dimension $N_{\mathrm g}$. For affine linear operators this is achieved by offline precomputation. For nonlinear terms, direct evaluation usually requires pointwise operations on the full spatial grid, causing the online cost to scale with $N_{\mathrm g}$. DEIM addresses this issue by approximating the nonlinear term as
\begin{equation}\label{DEIM-approximation}
    \V f(\mu) \approx \M \Psi \V c(\mu),
\end{equation}
where $\M\Psi$ is the DEIM basis. Provided that $\M P^\mathrm T\M\Psi$ is nonsingular, the coefficient vector $\V c(\mu)$ is determined by solving the square interpolation system
\begin{equation}
    \M P^\mathrm{T} \V f(\mu) = \M P^\mathrm{T} \M \Psi \V c(\mu).
\end{equation}
The matrix $\M P$ selects the interpolation indices in the discretized nonlinear term. It is defined as
\begin{equation}
    \M P = [\V i_{I_1}, \V i_{I_2}, \dots, \V i_{I_{N_{\mathrm{d}}}}],
\end{equation}
where $\V i_{I_k}\in\Set R^{N_{\mathrm g}}$ is the $I_k$-th canonical basis vector.

The DEIM basis is constructed from nonlinear snapshots $\M F=[\V f_1,\dots,\V f_{N_{\mathrm s}}]$. POD is applied to $\M F$, and the first $N_{\mathrm d}$ retained modes $[\V\psi_1,\dots,\V\psi_{N_{\mathrm d}}]$ form $\M\Psi$.
The interpolation indices $\{I_k\}_{k=1}^{N_{\mathrm d}}$ are selected iteratively by the standard DEIM residual criterion:
\begin{equation}
    I_k = \mathop{\mathrm{argmax}}_{i=1, \dots, N_{\mathrm{g}}} |\V i_i^\mathrm{T}(\V \psi_k - \M \Psi^{(k-1)} \V c^{(k)})|, \quad k=1, \dots, N_{\mathrm{d}},
\end{equation}
where $\M \Psi^{(k)} = [\V \psi_1, \dots, \V \psi_k]$, and $\V c^{(k)}$ is the solution to:
\begin{equation}
    (\M P^{(k-1)})^\mathrm{T} \M \Psi^{(k-1)} \V c^{(k)} = (\M P^{(k-1)})^\mathrm{T} \V \psi_k, \quad k=1, \dots, N_{\mathrm{d}}.
\end{equation}
Here, $\M P^{(k)}$ contains the first $k$ selected columns of $\M P$:
\begin{equation}
    \M P^{(k)} = [\V i_{I_1}, \V i_{I_2}, \dots, \V i_{I_k}], \quad k=1, \dots, N_{\mathrm{d}}.
\end{equation}
The resulting DEIM approximation is
\begin{equation}\label{eqn:DEIM_approximation}
    \V f(\mu) \approx \M \Psi (\M P^\mathrm{T} \M \Psi)^{-1} \M P^\mathrm{T} \V f(\mu).
\end{equation}
The term $\M \Psi (\M P^\mathrm{T} \M \Psi)^{-1}$ can be precomputed during the offline stage. Consequently, evaluating the nonlinear term in the reduced equations only requires the entries selected by $\M P^\mathrm{T} \V f(\mu)$; this removes the full-grid nonlinear evaluation from the online reduced solve, although full-state reconstruction or postprocessing still scales with $N_{\mathrm{g}}$. The detailed DEIM procedure is summarized in Algorithm~\ref{alg2:DEIM}.

\begin{algorithm}[ht]
    \caption{Discrete Empirical Interpolation Method (DEIM)}
    \label{alg2:DEIM}
    \begin{algorithmic}[1]
        \REQUIRE DEIM basis $\M \Psi = [\V \psi_1, \dots, \V \psi_{N_{\mathrm{d}}}] \in \Set{R}^{N_{\mathrm{g}} \times N_{\mathrm{d}}}$.
        \ENSURE Interpolation matrix $\M P = [\V i_{I_1}, \V i_{I_2}, \dots, \V i_{I_{N_{\mathrm{d}}}}] \in \Set{R}^{N_{\mathrm{g}} \times N_{\mathrm{d}}}$.
        \STATE Initialization for DEIM:
        \[
            I_1 = \mathop{\mathrm{argmax}}_{i=1, \dots, N_{\mathrm{g}}} |\V i_{i}^\mathrm{T} \V \psi_1|.
        \]
        Set $\M \Psi^{(1)} = [\V \psi_1]$ and $\M P^{(1)} = [\V i_{I_1}]$.
        \FOR{$k=2$ \TO $N_{\mathrm{d}}$} 
            \STATE Solve for $\V c^{(k)}$ from the linear system:
            \[
                (\M P^{(k-1)})^\mathrm{T} \M \Psi^{(k-1)} \V c^{(k)} = (\M P^{(k-1)})^\mathrm{T} \V \psi_k.
            \]
            \STATE Identify the spatial index $I_k$ where the approximation error is maximized:
            \[
                I_k = \mathop{\mathrm{argmax}}_{i=1, \dots, N_{\mathrm{g}}} |\V i_{i}^\mathrm{T}(\V \psi_k - \M \Psi^{(k-1)} \V c^{(k)})|.
            \]
            \STATE Update the basis and selection matrix:
            $\M \Psi^{(k)} = [\M \Psi^{(k-1)}, \V \psi_k]$, $\M P^{(k)} = [\M P^{(k-1)}, \V i_{I_k}]$.
        \ENDFOR
        \STATE Set $\M P = \M P^{(N_{\mathrm{d}})}$.
    \end{algorithmic}
\end{algorithm}

The DEIM basis is also snapshot-dependent, so standard POD/DEIM can inherit the same data-scarcity limitations as the solution basis. We apply SS-POD to the nonlinear snapshots as well. The augmented index set $\Set{I}^*_{\mathrm{D}}$ follows the same energy-based construction as $\Set{I}^*$, with solution snapshots $\M U$ replaced by nonlinear snapshots $\M F$. Algorithm~\ref{alg3:DEIM/ROSG/brief} gives the resulting SS-POD/DEIM procedure.

\begin{algorithm}[ht]
    \caption{SS-POD/DEIM for Nonlinear Approximation with Data Scarcity}
    \label{alg3:DEIM/ROSG/brief}
    \begin{algorithmic}[1]
        \REQUIRE Snapshots of the nonlinear term $\M F$ or solutions $\M U$, number of subspaces $N_{\mathrm{sub}}^\mathrm{D}$, and maximum spectral basis functions $N_{\mathrm{max}}^\mathrm{D}$.
        \ENSURE Augmented DEIM basis $\M \Psi_{\mathrm{Aug}}$ and interpolation matrix $\M P$.
        \STATE Generate $\M F$ from $\M U$ or collect snapshots of the nonlinear term $\M F$ during simulation.
        \STATE Invoke Algorithm~\ref{alg1:ROSG} using inputs: $\M U = \M F$, $N_{\mathrm{sub}} = N_{\mathrm{sub}}^\mathrm{D}$, and $N_{\mathrm{max}} = N_{\mathrm{max}}^\mathrm{D}$. 
        \STATE Obtain the augmented basis $\M \Psi_{\mathrm{Aug}}$.
        \STATE Invoke Algorithm~\ref{alg2:DEIM} with $\M \Psi = \M \Psi_{\mathrm{Aug}}$ to compute $\M P$.
        \STATE Return $\M \Psi_{\mathrm{Aug}}$ and $\M P$.
    \end{algorithmic}
\end{algorithm}

\subsection{Error and Complexity Analysis}
\subsubsection{Snapshot Reconstruction Error}
\begin{theorem}[Snapshot Reconstruction Error of SS-POD]
\label{thm:recons_error}
Let $\M U = [\V u_1,\ldots,\V u_{N_{\mathrm{s}}}] \in \Set{R}^{N_{\mathrm{g}}\times N_{\mathrm{s}}}$ be a snapshot matrix. Define the matrix norm induced by the discrete $\Set V$-inner product as
\[
    \norm{\M X}_{\Set V,\mathrm{F}}^2 := \sum_{i=1}^{N_{\mathrm{s}}} \norm{\V x_i}_{\Set V}^2,
    \qquad \M X=[\V x_1,\ldots,\V x_{N_{\mathrm{s}}}].
\]
Let $\M \Pi_n$ be the $\Set V$-orthogonal projection onto $\Set E_n$, let $\M U^{(n)}=\M \Pi_n\M U$, and let
\[
    \M U_{\Set E}=\sum_{n=1}^{N_{\mathrm{sub}}}\M U^{(n)},\qquad
    \M R_{\Set E}=\M U-\M U_{\Set E}.
\]
For each $n$, let $\M U_{\mathrm r}^{(n)}$ be the rank-truncated POD approximation of $\M U^{(n)}$ in $\Set E_n$, and set
\[
    \M U_{\mathrm r}=\sum_{n=1}^{N_{\mathrm{sub}}}\M U_{\mathrm r}^{(n)}.
\]
Suppose that the local truncation satisfies
\begin{equation}\label{eqn:recons_err_in_subspace}
    \norm{\M U^{(n)}-\M U_{\mathrm r}^{(n)}}_{\Set V,\mathrm{F}}^2
    \le \epsilon \norm{\M U^{(n)}}_{\Set V,\mathrm{F}}^2,
    \qquad n=1,\ldots,N_{\mathrm{sub}},
\end{equation}
for a prescribed tolerance $\epsilon>0$. Then
\begin{equation}\label{eqn:error_bound_weighted}
    \norm{\M U-\M U_{\mathrm r}}_{\Set V,\mathrm{F}}^2
    \le
    \norm{\M R_{\Set E}}_{\Set V,\mathrm{F}}^2
    + \epsilon \norm{\M U_{\Set E}}_{\Set V,\mathrm{F}}^2.
\end{equation}
In particular, if $\M R_{\Set E}=\M 0$, then
\begin{equation}\label{eqn:error_bound}
    \norm{\M U-\M U_{\mathrm r}}_{\Set V,\mathrm{F}}^2
    \le \epsilon \norm{\M U}_{\Set V,\mathrm{F}}^2.
\end{equation}
Moreover, if the discrete $\Set V$-norm and Euclidean norm satisfy
\[
    m_V\norm{\V v}_2^2\le \norm{\V v}_{\Set V}^2\le M_V\norm{\V v}_2^2,
    \qquad 0<m_V\le M_V<\infty,
\]
then,when $\M R_{\Set E}=\M 0$,
\begin{equation}\label{eqn:error_bound_euclidean}
    \norm{\M U-\M U_{\mathrm r}}_{\mathrm{F}}^2
    \le \frac{M_V}{m_V}\,\epsilon \norm{\M U}_{\mathrm{F}}^2.
\end{equation}
\end{theorem}

\begin{remark}
This result concerns reconstruction of the available snapshot matrix; it is not an a priori error bound for out-of-sample parameters. The residual $\M R_{\Set E}$ records the part of the snapshots not represented by the finite spectral prior. If this residual is small but nonzero, Eq.~\eqref{eqn:error_bound_weighted} should be used instead of the simplified bound in Eq.~\eqref{eqn:error_bound}. The constant $M_V/m_V$ is a norm-equivalence constant for the chosen discretization and inner product; it is not a data-driven quantity.
\end{remark}

\begin{proof}
Because $\M \Pi_n$ are $\Set V$-orthogonal projections onto pairwise orthogonal subspaces, $\M R_{\Set E}$ is $\Set V$-orthogonal to every $\Set E_n$, and the local errors $\M U^{(n)}-\M U_{\mathrm r}^{(n)}$ lie in mutually orthogonal subspaces. Therefore,
\begin{equation}
\begin{aligned}
    \norm{\M U-\M U_{\mathrm r}}_{\Set V,\mathrm{F}}^2
    &= \norm{\M R_{\Set E}
    + \sum_{n=1}^{N_{\mathrm{sub}}}
    (\M U^{(n)}-\M U_{\mathrm r}^{(n)})}_{\Set V,\mathrm{F}}^2  \\
    &= \norm{\M R_{\Set E}}_{\Set V,\mathrm{F}}^2
    + \sum_{n=1}^{N_{\mathrm{sub}}}
    \norm{\M U^{(n)}-\M U_{\mathrm r}^{(n)}}_{\Set V,\mathrm{F}}^2 .
\end{aligned}
\end{equation}
Using the assumed local truncation estimate in Eq.~\eqref{eqn:recons_err_in_subspace} gives
\begin{equation}
    \norm{\M U-\M U_{\mathrm r}}_{\Set V,\mathrm{F}}^2
    \le
    \norm{\M R_{\Set E}}_{\Set V,\mathrm{F}}^2
    + \epsilon \sum_{n=1}^{N_{\mathrm{sub}}}
    \norm{\M U^{(n)}}_{\Set V,\mathrm{F}}^2.
\end{equation}
Since the projected snapshot blocks $\M U^{(n)}$ are also mutually $\Set V$-orthogonal,
\begin{equation}
    \sum_{n=1}^{N_{\mathrm{sub}}}\norm{\M U^{(n)}}_{\Set V,\mathrm{F}}^2
    = \norm{\M U_{\Set E}}_{\Set V,\mathrm{F}}^2,
\end{equation}
which proves Eq.~\eqref{eqn:error_bound_weighted}. If $\M R_{\Set E}=\M 0$, then $\M U_{\Set E}=\M U$ and Eq.~\eqref{eqn:error_bound} follows.

Finally, the norm equivalence
$m_V\norm{\V v}_2^2\le\norm{\V v}_{\Set V}^2\le M_V\norm{\V v}_2^2$
implies
\[
    \norm{\M X}_{\mathrm F}^2\le m_V^{-1}\norm{\M X}_{\Set V,\mathrm F}^2,
    \qquad
    \norm{\M X}_{\Set V,\mathrm F}^2\le M_V\norm{\M X}_{\mathrm F}^2.
\]
Applying these two inequalities to Eq.~\eqref{eqn:error_bound} yields Eq.~\eqref{eqn:error_bound_euclidean}.
\end{proof}

\begin{remark}
    The local estimate in Eq.~\eqref{eqn:recons_err_in_subspace} follows from the Schmidt-Eckart-Young theorem when the POD/SVD is performed in coordinates consistent with the discrete $\Set V$-inner product. The Euclidean version follows after the norm-equivalence conversion above.
\end{remark}

\subsubsection{Discussions about the Online Approximation Error}\label{sec:error_online}
The snapshot reconstruction estimate above does not by itself guarantee accuracy for a new parameter value $\mu$ outside the training set. For an out-of-sample solution $\V u=\V u(\mu)$, the SS-POD approximation error can be decomposed across the spectral subspaces as
\begin{equation}
\begin{aligned}
    \norm{ \V u - \M \Phi \V a } 
    &= \norm{ \sum_{n=1}^{N_{\mathrm{sub}}} \V u^{(n)} - \sum_{n=1}^{N_{\mathrm{sub}}} \M \Phi^{(n)} \V a^{(n)} } \\
    &\le \sum_{n=1}^{N_{\mathrm{sub}}} \norm{ \V u^{(n)} - \M \Phi^{(n)} \V a^{(n)} }.
\end{aligned}
\end{equation}
This relation is an error decomposition, not a closed a priori bound. The terms on the right depend on how well each local POD space, learned from projected training snapshots, represents the corresponding component of the new solution. The energy-balanced partition reduces the risk that a single global POD step discards low-energy spectral components. It does not remove the usual accuracy requirements: the spectral prior must approximate the relevant solution class, the snapshot set must contain representative parameter information, and the Galerkin or DEIM online solve must be accurate enough for the target tolerance.

\subsubsection{Complexity Analysis and Parallel Computing for the Offline Procedure}\label{sec:complexity}

Let $N_{\mathrm{basis}}$ denote the total number of reduced basis functions used by a given method. For SS-POD, $N_{\mathrm{basis}}=\sum_{n=1}^{N_{\mathrm{sub}}}N_{\mathrm p}^{(n)}$; for standard POD and spectral-Galerkin methods, it is $N_{\mathrm p}$ and $N_{\mathrm{max}}$, respectively.

The main computational costs are summarized below. The estimates are intended as order-of-magnitude costs for the implementations used in this work; they do not account for all possible sparsity, tensor-product, or fast-transform optimizations.
\begin{itemize}
    \item For some properly chosen Fourier or Chebyshev bases, the spectral coefficients can be computed with fast transforms, with cost $\mathcal{O}(N_{\mathrm{s}}N_{\mathrm{g}}\log N_{\mathrm{g}})$ when the discretization supports such transforms.
    \item  The identification of the energy-balanced partition $\Set{I}^*$ has a complexity of $\mathcal{O}(N_{\mathrm{s}} N_{\mathrm{g}})$.
    \item Computing the retained POD modes by truncated SVD costs approximately $\mathcal{O}(N_{\mathrm{s}}N_{\mathrm{g}}N_{\mathrm{basis}})$, up to implementation-dependent constants and ignoring discarded near-null modes.
    \item The precomputation of reduced mass and stiffness matrices scales with two transform-based terms: $\mathcal{O}(N_{\mathrm{basis}}^2N_{\mathrm{g}})$ for reduced matrix contractions and $\mathcal{O}(N_{\mathrm{basis}}N_{\mathrm{g}}\log N_{\mathrm{g}})$ for transform operations. Some spectral-Galerkin matrix entries can instead be evaluated analytically.
    \item For $N_\mu$ new parameter values, online affine assembly scales as $\mathcal{O}(N_\mu N_{\mathrm{basis}}^2)$, with up to $\mathcal{O}(N_\mu N_{\mathrm{basis}}^3)$ additional cost for dense direct reduced solves.
    \item Reconstructing full-order solution fields from reduced coefficients requires $\mathcal{O}(N_{\mathrm{g}}N_{\mu}N_{\mathrm{basis}})$ operations.
\end{itemize}
The local POD computations for the projected matrices $\M U^{(n)}$ are independent across subspaces and can be parallelized. We implemented this offline step in MATLAB with the Parallel Computing Toolbox.

\subsection{Implementation Details and Practical Considerations}
\subsubsection{Boundary Conditions: Dirichlet BC and Choice of Test Functions} \label{choice_of_test_function}
Boundary conditions require special care in projection-based ROMs because a basis constructed from snapshots or spectral modes does not automatically satisfy the desired test-space constraints. For inhomogeneous Dirichlet conditions, standard remedies include the control function method~\cite{graham1999optimal,gunzburger2007reduced} and modified basis methods~\cite{gunzburger2007reduced}. In Chebyshev-Galerkin discretizations, boundary-compatible trial and test functions can be constructed following Shen~\cite{shen1995efficient} or Heinrichs~\cite{heinrichs1991stabilized}. In the Dirichlet examples below, we project the trial basis functions onto the homogeneous boundary-compatible Chebyshev subspace. Let $\V\chi_j=\V T_j-\V T_{j-2}$ for $j=2,\ldots,N_{\mathrm{max}}$. The test function associated with $\V\phi_{\mathrm{Aug},i}$ is written as
\begin{equation}\label{eqn:test_function}
   \tilde{\V \phi}_{\mathrm{Aug},i} = \sum_{j=2}^{N_{\mathrm{max}}} c_{ij}\V\chi_j,
\end{equation}
where the coefficients solve the Gram system
\begin{equation}\label{eqn:test_function_gram}
    \sum_{\ell=2}^{N_{\mathrm{max}}} \av{\V\chi_\ell,\V\chi_j}_{\Set V} c_{i\ell}
    = \av{\V \phi_{\mathrm{Aug},i},\V\chi_j}_{\Set V},
    \qquad j=2,\ldots,N_{\mathrm{max}}.
\end{equation}
If $\{\V\chi_j\}$ is first orthonormalized in the $\Set V$-inner product, this reduces to the simple coefficient formula given by direct inner products. The resulting $\tilde{\V\phi}_{\mathrm{Aug},i}$ satisfies the homogeneous boundary condition by construction. This test-space construction can be interpreted as a projection of the SS-POD trial basis onto
\begin{equation}
    \mathrm{span} \{ \V T_j - \V T_{j-2} \}_{j=2}^{N_{\mathrm{max}}}.
\end{equation}
For a 1D problem with Dirichlet boundary conditions, this construction leaves two boundary degrees of freedom in the reduced coefficient system. We close the system by imposing the two boundary constraints
\begin{equation}
   \sum_i a_i(\mu)\V\phi_{\mathrm{Aug},i}(x)\big|_{x=\pm1}=0.
\end{equation}

\subsubsection{Threshold for Spurious Basis Functions} \label{threshold_junk_basis}
Truncation thresholds are used to remove numerically insignificant modes and to keep $\sum_{n=1}^{N_{\mathrm{sub}}}N_{\mathrm p}^{(n)}\le N_{\mathrm{max}}$. Modes associated with very small singular values contribute little to the snapshot reconstruction and can worsen conditioning in the reduced systems.
We use $\mathrm{tol}_{\mathrm p}$ for standard POD and $\mathrm{tol}_{\mathrm p}^{(n)}$ for the local SS-POD truncations; only modes whose singular values exceed the prescribed tolerance are retained.
Analogously, thresholds $\mathrm{tol}_{\mathrm d}$ and $\mathrm{tol}_{\mathrm d}^{(n)}$ are used for DEIM and SS-POD/DEIM. 
All the parameter values used in the reported numerical experiments are listed in Table~\ref{tab:para}.

\section{Numerical Results}\label{sec:numerical_result}

\subsection{Technical Details}\label{Tech_detail}

\subsubsection{Notations and Parameter Setup}

Let $\M U_{\mathrm{test}}=[\V u(\mu_1),\dots,\V u(\mu_{N_\mu})]$ collect the full-order test solutions, and let $\hat{\M U}_{\mathrm{test}}$ be the corresponding ROM approximation. The error matrix and relative error, measured in the vectorized $\ell^2$ norm equivalently the Frobenius norm, are
\begin{align}
    \M \epsilon &= \M U_{\mathrm{test}} - \hat{\M U}_{\mathrm{test}}, \\
   \epsilon^{\mathrm{r}} &= \frac{\norm{\M \epsilon}_{\mathrm{F}}}{\norm{\M U_{\mathrm{test}}}_{\mathrm{F}}}.
\end{align}
When POD and SS-POD errors are reported separately, we use $\M\epsilon_{\mathrm P}$ and $\M\epsilon_{\mathrm S}$, respectively.

For nonlinear examples, we also report the DEIM approximation error for the nonlinear snapshot matrix $\M F_{\mathrm{test}}=[\V f(\mu_1),\dots,\V f(\mu_{N_\mu})]$:
\begin{equation}
   \epsilon_{\mathrm{D}}^{\mathrm{r}} = \frac{\norm{\M \Psi (\M P^\mathrm{T} \M \Psi)^{-1} \M P^\mathrm{T} \M F_{\mathrm{test}} - \M F_{\mathrm{test}}}_{\mathrm{F}}}{\norm{\M F_{\mathrm{test}}}_{\mathrm{F}}},
\end{equation}
The number of reduced basis functions used for the solution approximation is denoted by $N_{\mathrm{basis}}$, as in Sect.~\ref{sec:complexity}. For nonlinear terms approximated by POD/DEIM or SS-POD/DEIM, the corresponding DEIM basis size is denoted by $N_{\mathrm{basis}}^{\mathrm D}$.

For nonlinear time-dependent Galerkin systems, the reduced coefficients are obtained by solving nonlinear algebraic systems at each time step. We use the Levenberg-Marquardt algorithm, initializing $\V a^{t+1}$ with $\V a^t$, and set the stopping tolerances to $\mathrm{tol}_{\mathrm{fun}}=10^{-11}$ and $\mathrm{tol}_{\mathrm{grad}}=10^{-11}$.
\begin{table}[ht]
\caption{Parameter settings for the numerical experiments.}
\label{tab:para}
\centering
\begin{tabular}{llll}
\hline\noalign{\smallskip}
 & Poisson & Helmholtz & Heat \\
\noalign{\smallskip}\hline\noalign{\smallskip}
$N_{\mathrm{max}}$ & 41$\times$41 & 31$\times$31 & 400 \\
$\mathrm{tol}_{\mathrm{p}}^{n}$ & $1\times 10^{-9}$ & $1\times 10^{-9}$ & $1\times 10^{-7}$ \\
\noalign{\smallskip}\hline\noalign{\smallskip}
 & 1D Allen-Cahn & 2D Allen-Cahn & Laplace-Beltrami \\
\noalign{\smallskip}\hline\noalign{\smallskip}
$N_{\mathrm{max}}$ & 150 & 80$\times$80 & 256 ($l=15$) \\
$\mathrm{tol}_{\mathrm{p}}^{n}$ & $1\times 10^{-11}$ & $1\times 10^{-6}$ & $1\times 10^{-10}$ \\
\noalign{\smallskip}\hline
\end{tabular}
\end{table}
The MATLAB function \texttt{svds} is used for truncated SVD computations. Since the snapshot counts are small in the reported data-scarce tests, the requested SVD rank is initially set to $N_{\mathrm s}$, and modes with negligible singular values are subsequently removed by the tolerances in Table~\ref{tab:para}. For the nonlinear examples, we set $N_{\mathrm{sub}}^{\mathrm D}=N_{\mathrm{sub}}$ as an empirical choice so that the nonlinear basis receives the same spectral-subspace resolution as the solution basis.

In randomized tests, snapshots are selected using MATLAB's default random number generator. In the sphere example, the spherical harmonic transform library of~\cite{politis2016microphone} is used for spherical transforms and numerical integration.

\subsubsection{Temporal Discretization Scheme}

For time-dependent problems, the temporal interval is discretized uniformly with step size $\Delta t=t_{i+1}-t_i$. Because the purpose of the experiments is to compare spatial reduced bases, we use sufficiently stable temporal discretizations to limit the influence of time-stepping error. Unless otherwise stated, the Galerkin system is discretized with the second-order Adams-Moulton method, i.e., the trapezoidal rule, which is A-stable for linear autonomous problems.
For the 2D Allen-Cahn example, we additionally use an operator splitting scheme, since this is a standard and efficient treatment of the reaction and diffusion components.

\subsection{Linear, Time-independent Cases: 2D Poisson and Helmholtz Equations}

\subsubsection{2D Poisson Equation}\label{numerical:Multi_Poisson}

We first consider a 2D Poisson equation with a parametrized multiscale source:
\begin{equation}\label{eqn:poisson}
    \left \{
    \begin{array}{ll}
    -\Delta u(x,y;a_k) = f(x,y;a_k) &  (x,y) \in \Omega \setminus \partial\Omega \\
    u(x,y;a_k) = 0 & (x,y) \in  \partial\Omega
    \end{array}
    \right.
\end{equation}
where $\Omega=[-1,1]^2$ and
\begin{equation}
    f(x,y;a_k) = \prod_{k=1}^{N} (1+\frac{1}{2}\cos(a_k \pi (x+y)))(1+\frac{1}{2}\sin(a_k \pi (x-3y))).
\end{equation}
The parameter $a_k$ varies in $[2^{k-2},1.5\times2^{k-2}]$, with $N=2$. Fig.~\ref{fig:source_ref_err_poisson} shows the source, the reference solution, and the error distributions for a limited-snapshot test. With $N_{\mathrm s}=15$ snapshots, standard POD gives a larger relative error than SS-POD for this test.

\begin{figure}[ht]
    \centering
    \includegraphics[width=1\textwidth]{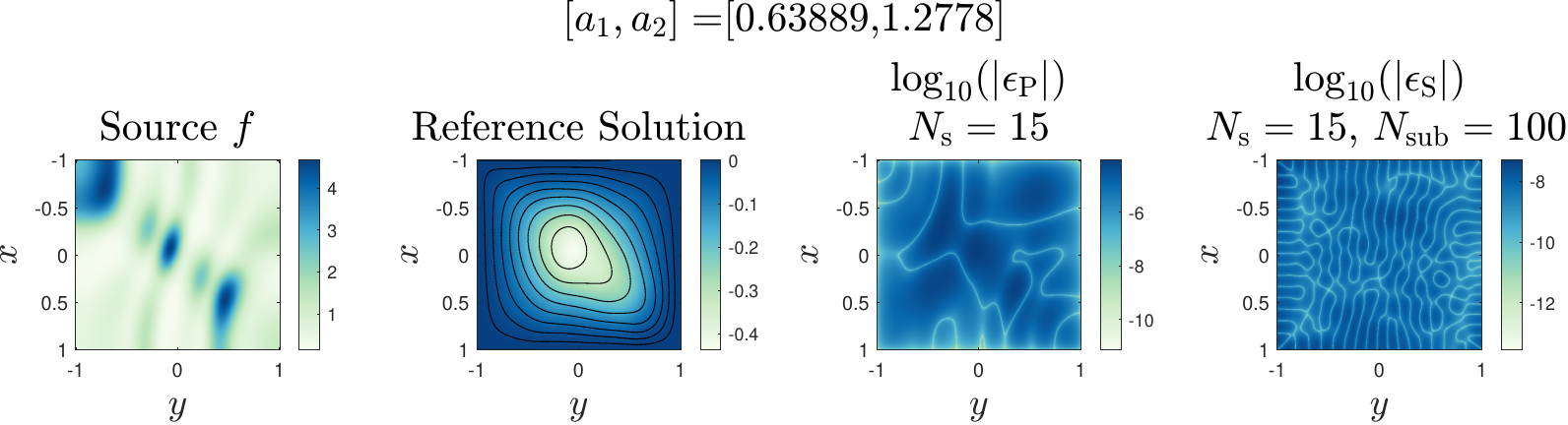}
    \caption{Poisson equation results: source $f$ (left), reference solution (middle-left), and error maps for POD (middle-right, $N_{\mathrm{basis}}=15, \epsilon_{\mathrm{P}}^\mathrm{r}=3.70\times 10^{-4}$) and SS-POD (right, $N_{\mathrm{basis}}=361, \epsilon_\mathrm{S}^\mathrm{r}=5.92\times 10^{-7}$). $N_{\mathrm{sub}}$ is the number of subspaces generated by Algorithm~\ref{alg4:I_ROSG}.}
    \label{fig:source_ref_err_poisson}
\end{figure}

For SS-POD, we use tensor-product Chebyshev functions
\begin{equation}\label{eqn:2d_cheb}
    T_{i,j}(x,y) = T_{i}(x) T_{j}(y).
\end{equation}
The two-dimensional indices $(i,j)$ are ordered by $k=\sqrt{i^2+j^2}$ so that the energy-balanced partition can be applied to a one-dimensional spectral ordering. The Galerkin projection yields
\begin{equation}\label{eqn:gov_poisson}
   -\left(\av{\tilde{\M \Phi}, \M \Phi_{xx}}_{\Set{V}} + \av{\tilde{\M \Phi}, \M \Phi_{yy}}_{\Set{V}}\right) \V a(\mu) = \av{\tilde{\M \Phi}, \V f(\mu)}_{\Set{V}},
\end{equation}
subject to the boundary condition $\M \Phi(I^{\mathrm{b}},:) \V a(\mu) = \V 0$. 

Fig.~\ref{fig:ROSG_numerical_tests}(a) compares the three reduced spaces. The Chebyshev-Galerkin curve gives the data-independent spectral baseline. Standard POD decreases the error at small basis sizes and then saturates in this limited-snapshot setting. SS-POD uses the spectral prior together with the snapshots and reaches $\epsilon_\mathrm{S}^\mathrm{r}=5.92\times 10^{-7}$ for the reported configuration.

\begin{figure}[ht]
    \centering 
    \begin{subfigure}[b]{0.32\textwidth} 
        \centering
        \includegraphics[width=\textwidth]{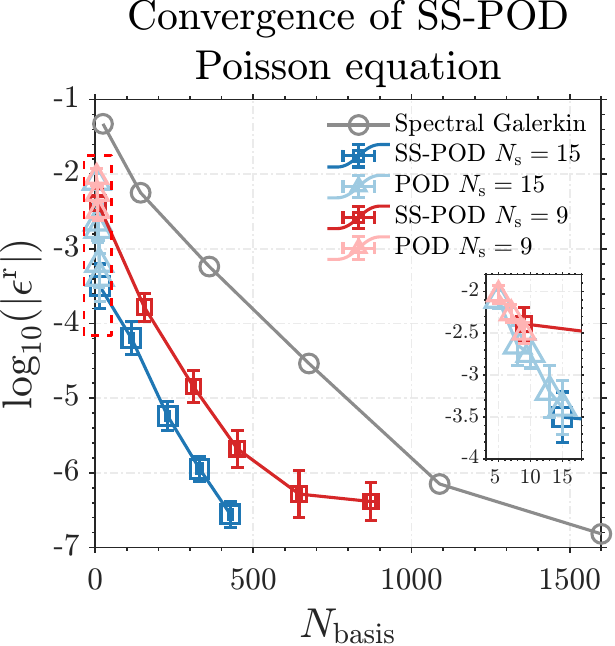}
        \caption{}
        \label{fig:sub1}
    \end{subfigure}
    \hfill 
    \begin{subfigure}[b]{0.32\textwidth}
        \centering
        \includegraphics[width=\textwidth]{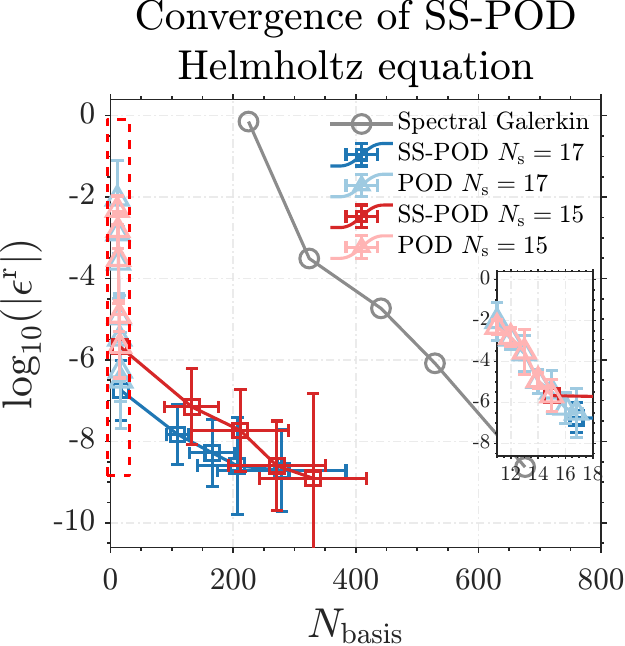}
        \caption{}
        \label{fig:sub2}
    \end{subfigure}
    \hfill 
    \begin{subfigure}[b]{0.32\textwidth}
        \centering
        \includegraphics[width=\textwidth]{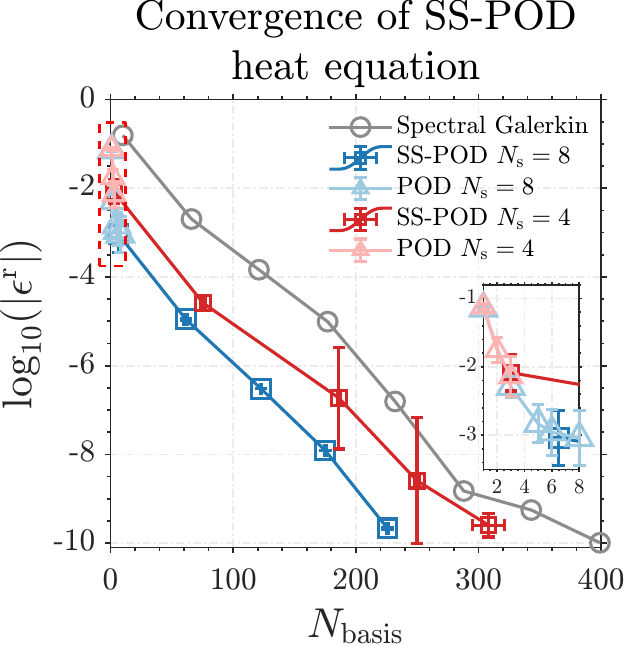}
        \caption{}
        \label{fig:sub3}
    \end{subfigure}

    \vspace{0.5cm} 

    \begin{subfigure}[b]{0.32\textwidth}
        \centering
        \includegraphics[width=\textwidth]{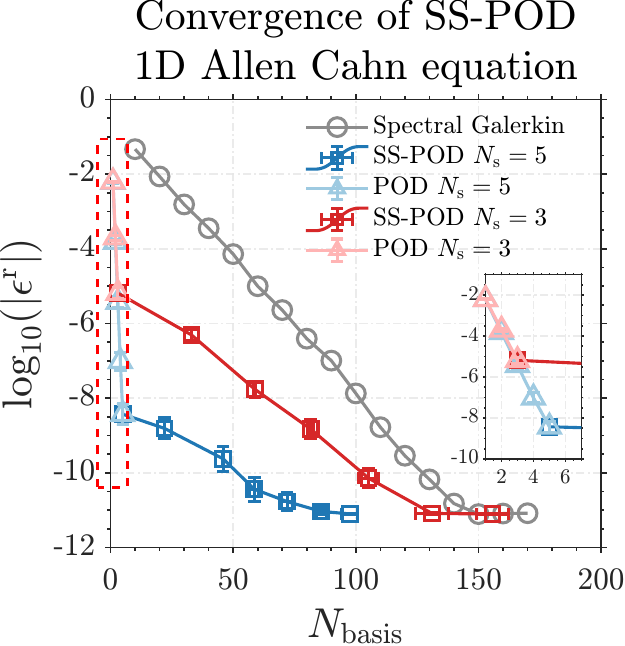}
        \caption{}
        \label{fig:sub4} 
    \end{subfigure}
    \hfill
    \begin{subfigure}[b]{0.32\textwidth}
        \centering
        \includegraphics[width=\textwidth]{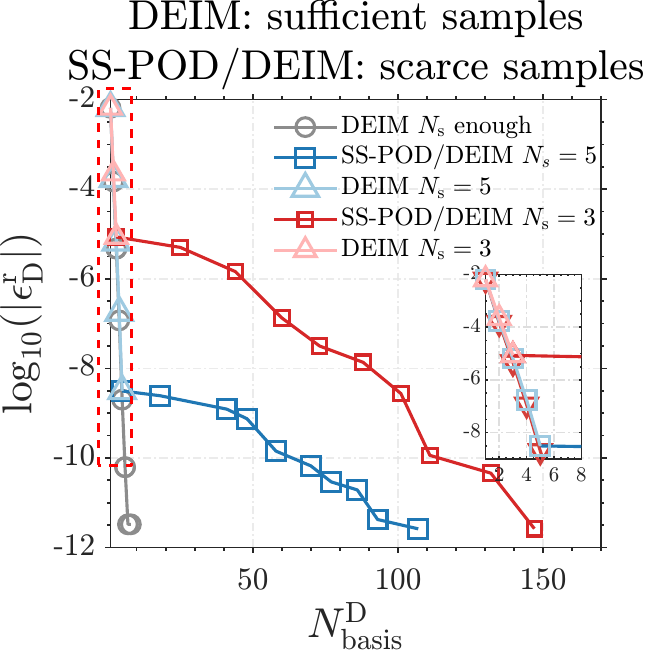}
        \caption{}
        \label{fig:sub5}
    \end{subfigure}
    \hfill
    \begin{subfigure}[b]{0.32\textwidth}
        \centering
        \includegraphics[width=\textwidth]{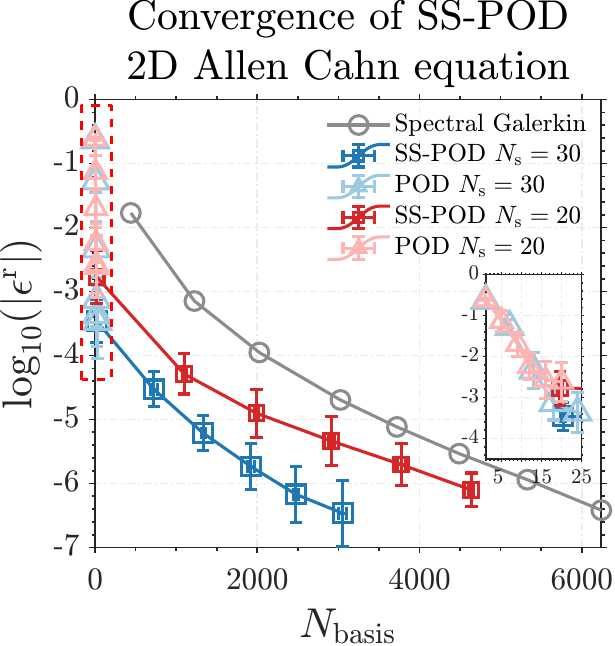}
        \caption{}
        \label{fig:sub6}
    \end{subfigure}

    \caption{Comparison of spectral-Galerkin, POD, and SS-POD results across the numerical examples. Panel (e) reports the SS-POD/DEIM approximation of the nonlinear term in the 1D Allen-Cahn equation. The snapshots are randomly selected.} 
    \label{fig:ROSG_numerical_tests}
\end{figure}

\subsubsection{Helmholtz Equation}\label{numerical:Helm}
The second elliptic benchmark is the 2D Helmholtz equation
\begin{equation}\label{eq:Helmholtz}
    \left \{
    \begin{array}{ll}
    \Delta u + k^2 u = f(x,y) & (x,y) \in \Omega \setminus \partial\Omega \\
    u(x,y) = 0 & (x,y) \in  \partial\Omega
    \end{array}
    \right. 
\end{equation}
with wavenumber $k\in[8,10]$ and source $f(x,y)=\exp(-10[(y-1)^2+(x-0.5)^2])$. The interval includes values close to eigenfrequencies of the homogeneous Dirichlet problem, for example $k=\frac{\pi}{2}\sqrt{3^2+5^2}\approx9.1592$~\cite{trefethen2000spectral}. The reference solution at $k=9.1515$, shown in Fig.~\ref{fig:sol_helm}, is close to the $(3,5)$ mode and therefore provides a more oscillatory test than the Poisson example.
\begin{figure}[ht]
    \centering
    \includegraphics[width=1\textwidth]{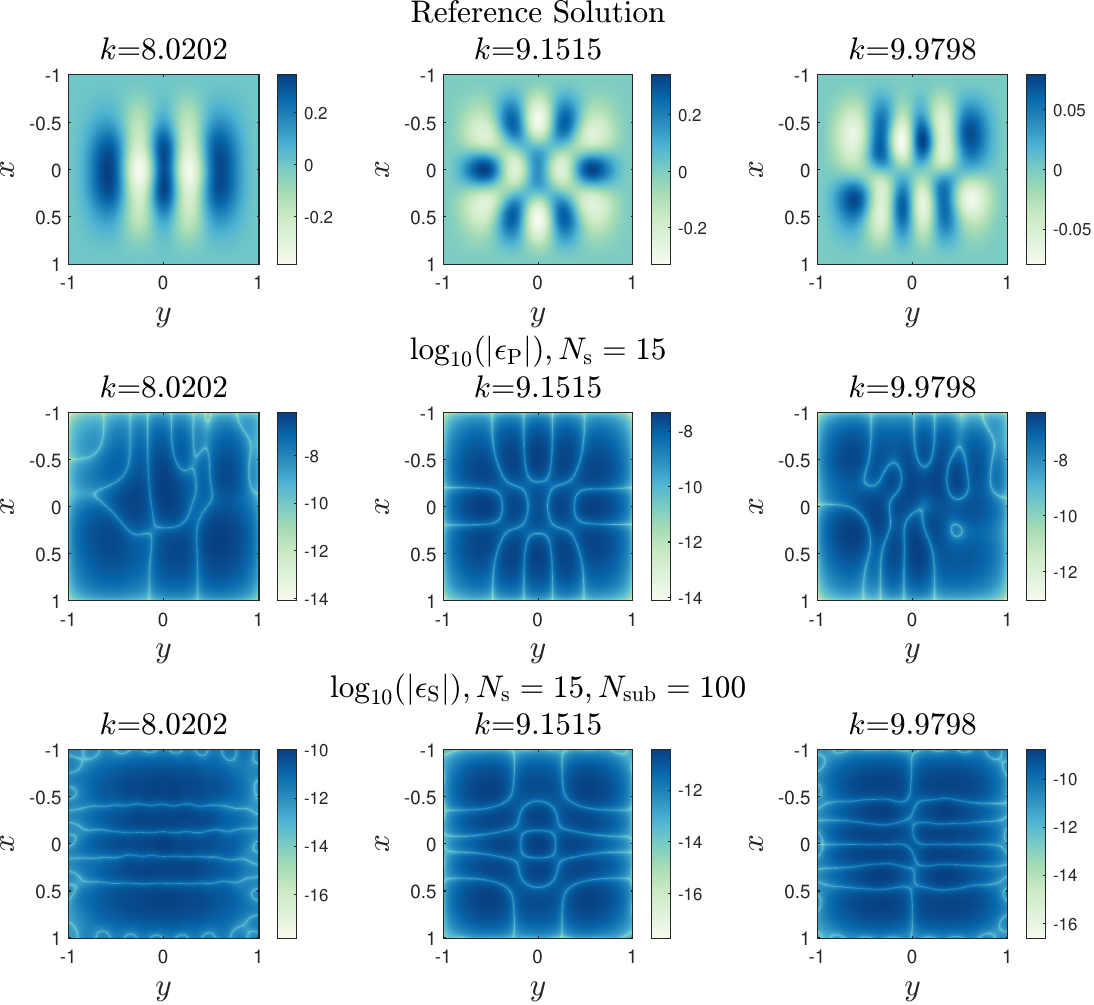}
    \caption{Helmholtz benchmarks: reference solution at $k=9.1515$ (top) and error distributions for standard POD (middle, $N_{\mathrm{basis}}=15$ and $\epsilon_{\mathrm{P}}^\mathrm{r}=2.91\times 10^{-7}$) vs. SS-POD (bottom, $N_{\mathrm{basis}}=229$ and $\epsilon_\mathrm{S}^\mathrm{r}=1.04\times 10^{-8}$).}
    \label{fig:sol_helm}
\end{figure}
The corresponding Galerkin ROM is
\begin{equation}
    (\av{\tilde{\M \Phi}, \M \Phi_{xx}}_{\Set{V}} + \av{\tilde{\M \Phi}, \M \Phi_{yy}}_{\Set{V}} + k^2\av{\tilde{\M \Phi}, \M \Phi}_{\Set{V}}) \V a(k) = \av{\tilde{\M \Phi}, \V f}_{\Set{V}}.
\end{equation}

Fig.~\ref{fig:ROSG_numerical_tests}(b) shows larger error variability for standard POD in the near-resonant regime with limited snapshots ($N_{\mathrm{s}}=15,17$). SS-POD reduces this variability in the tested configurations by separating the snapshot information across spectral subspaces.

\subsection{Linear, Time-dependent Case: 1D Heat Equation}\label{numerical:heat}

We next solve a 1D heat equation with a multiscale source:
\begin{equation}\label{eq:heat}
    \left \{
    \begin{array}{ll}
    \partial_t u(x,t) = \alpha \nabla^2 u(x,t) + f(x; a_k) & x \in (-1, 1), t \in [0,1] \\
    u(\pm 1, t) = 0, \quad u(x,0) = 0 &
    \end{array}
    \right.
\end{equation}
where $f(x;a_k)=\prod_{k=1}^{N}(1+\frac{1}{2}\cos(a_k\pi x))(1+\frac{1}{2}\sin(a_k\pi x))\times10^{-3}$ with $N=5$. The parameters are sampled from $[2^{k-1},1.5\times2^{k-1}]$. In this example the reduced basis is learned from temporal snapshots of a single reference trajectory, so data scarcity refers to sparse sampling in time rather than sparse sampling of a parameter grid. Fig.~\ref{fig:ref_sol_heat} illustrates the transient solution and the corresponding error maps.
\begin{figure}[ht]
    \centering
    \includegraphics[width=1\textwidth]{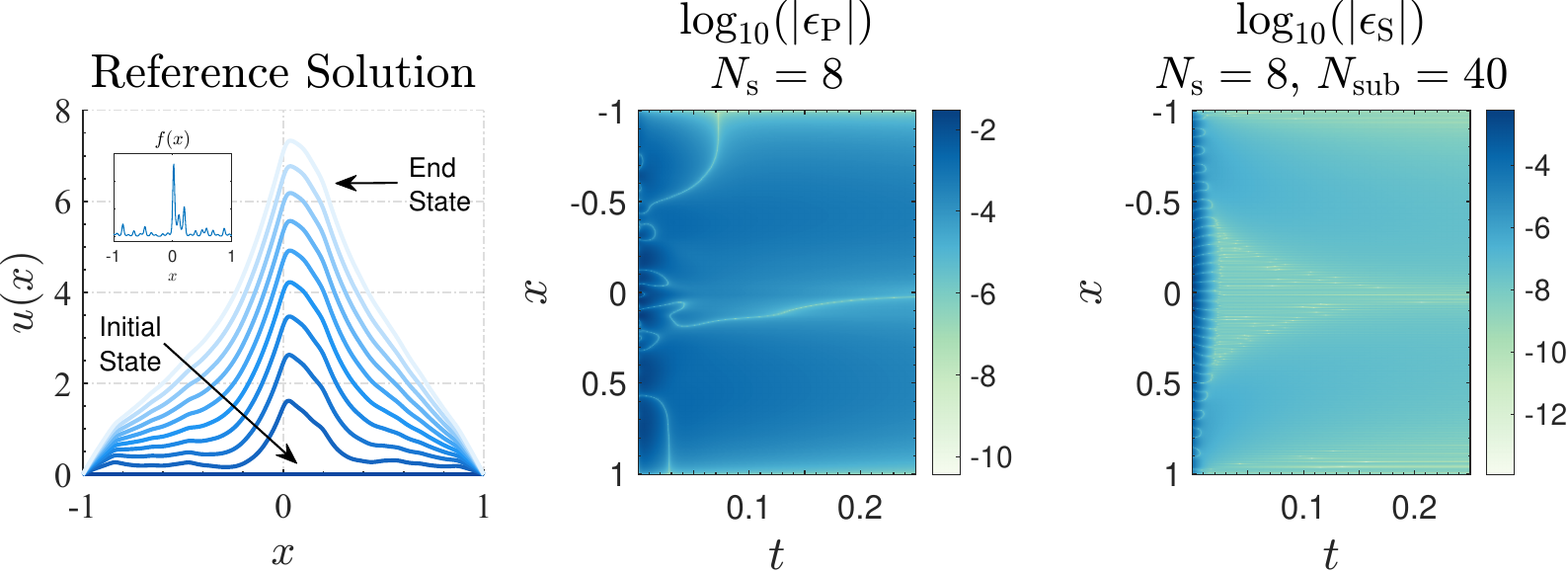}
    \caption{1D Heat equation benchmarks: reference evolution (left) and temporal error maps for standard POD ($N_{\mathrm{basis}}=15, \epsilon_{\mathrm{P}}^\mathrm{r}=9.44\times 10^{-4}$) vs. SS-POD ($N_{\mathrm{basis}}=229, \epsilon_\mathrm{S}^\mathrm{r}=5.98\times 10^{-5}$).}
    \label{fig:ref_sol_heat}
\end{figure}
The temporal discretization of the reduced system leads to
\begin{equation}
\begin{aligned}
    &\left(\av{\tilde{\M \Phi}, \M \Phi}_{\Set{V}}
    - \alpha \frac{\Delta t}{2}\av{\tilde{\M \Phi}, \M \Phi_{xx}}_{\Set{V}}\right)\V a^{t+1} \\
    &\qquad =
    \left(\av{\tilde{\M \Phi}, \M \Phi}_{\Set{V}}
    + \alpha \frac{\Delta t}{2}\av{\tilde{\M \Phi}, \M \Phi_{xx}}_{\Set{V}}\right)\V a^t
    + \Delta t\,\av{\tilde{\M \Phi}, \V f}_{\Set{V}} .
\end{aligned}
\end{equation}

Fig.~\ref{fig:ROSG_numerical_tests}(c) reports the case with $N_{\mathrm{s}}=4$ temporal snapshots. Increasing $N_{\mathrm{sub}}$ reduces the large errors caused by temporal under-sampling for this multiscale transient.

\subsection{Nonlinear Time-dependent Case: 1D Allen Cahn Equation and SS-POD/DEIM}\label{numerical:react_diff}

We now consider the 1D Allen-Cahn equation
\begin{equation}\label{eqn:reaction diffusion}
    \left \{
    \begin{array}{ll}
   \partial_t u = \frac{1}{L^2} \nabla^2 u + u - u^3, &  x\in (-1,+\infty), \ t \in [0,2], \\
    u(-1,t) = 0, \quad \lim_{x \to \infty} u(x,t) = 1,& t \in [0,2],\\
    u(x,0) = u_0(x), & x\in[-1,+\infty).
    \end{array}
    \right.
\end{equation}
with $L=30$ and $u_0(x)=\frac{a(x)-b(x)}{a(x)+b(x)+C}$, where $a(x)=\exp(\frac{\sqrt{2}L}{2}x+\frac{\sqrt{2}L}{2})$, $b(x)=\exp(-\frac{\sqrt{2}L}{2}x-\frac{\sqrt{2}L}{2})$, and $C=10^4$. The reference solution is $u(x,t)=\frac{a(x)-b(x)}{a(x)+b(x)+C\exp(-1.5t)}$. In the numerical implementation, the semi-infinite spatial domain is truncated to $[-1,1]$.
\begin{figure}[ht]
    \centering
    \includegraphics[width=1\textwidth]{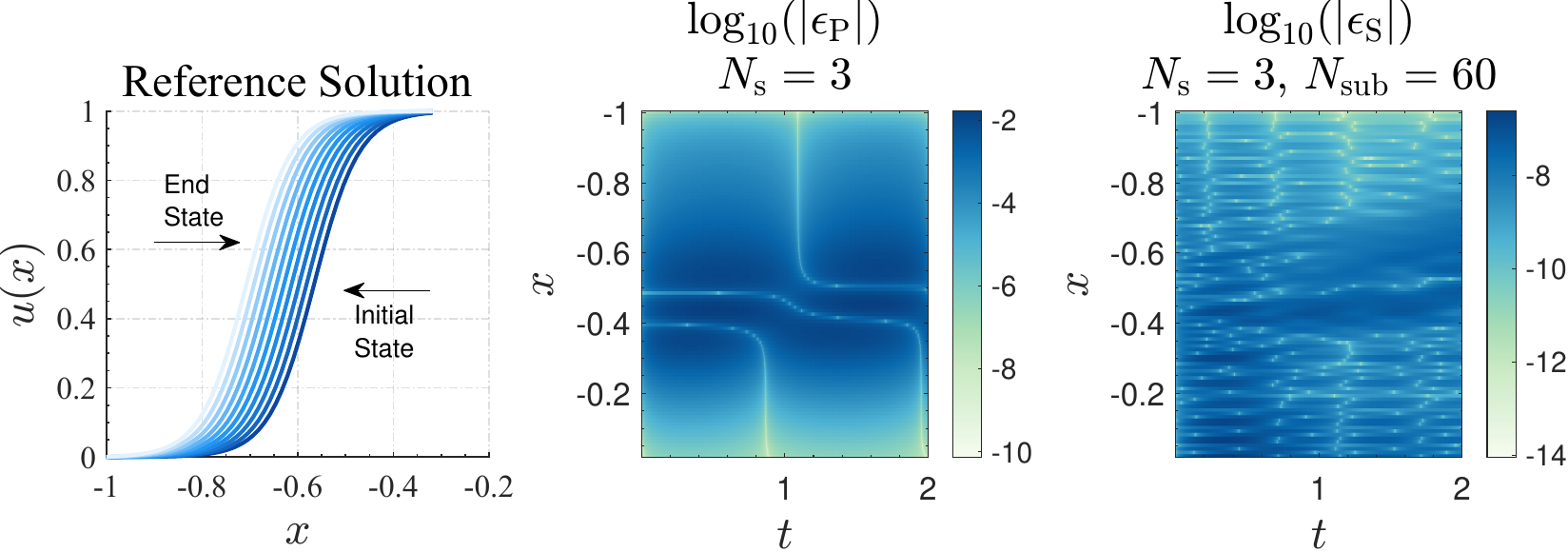}
    \caption{1D Allen-Cahn equation benchmark: reference evolution (left) and error comparisons. For the reported configuration, SS-POD/DEIM ($N_{\mathrm{basis}}=87$) gives $\epsilon_\mathrm{S}^\mathrm{r}=4.93\times 10^{-8}$, compared with $2.97\times 10^{-3}$ for standard POD/DEIM with $N_{\mathrm{basis}}=15$.}
    \label{fig:ref_sol_react_diff}
\end{figure}
The Galerkin reduced system is
\begin{equation}
    \av{\tilde{\M \Phi}, \M \Phi}_{\Set{V}} \frac{d \V a(t)}{d t} = \av{\tilde{\M \Phi}, \frac{\M \Phi_{xx}}{L^2} + \M \Phi}_{\Set{V}} \V a(t) - \av{\tilde{\M \Phi}, \V f(\V a)}_{\Set{V}}.
\end{equation}
The cubic nonlinear term is approximated by SS-POD/DEIM:
\begin{equation}\label{eqn:DEIM_appro_nonlinear}
    \V f(\V a) \approx \M \Psi(\M P^\mathrm{T} \M \Psi)^{-1} (\M P^\mathrm{T} \M \Phi \V a \odot \M P^\mathrm{T} \M \Phi \V a \odot \M P^\mathrm{T} \M \Phi \V a),
\end{equation}
where $\odot$ denotes the Hadamard product. Fig.~\ref{fig:ROSG_numerical_tests}(d) and (e) report the data-scarce cases ($N_{\mathrm{s}}=3,5$). SS-POD/DEIM preserves the DEIM-style reduced nonlinear evaluation and gives a lower relative error than the standard POD/DEIM baseline in these tests.

\subsection{Nonlinear, Time-dependent Case: 2D Allen-Cahn Equation}\label{numerical:AC}
The 2D Allen-Cahn equation reads:
\begin{equation}\label{eq:AC}
    \partial_t u = \epsilon^2 \Delta u + u - u^3, \quad (x,y) \in \Omega = [-1,1]^2.
\end{equation}
Periodic boundary conditions are imposed. The initial condition is generated from a Fourier series with frequencies below 4 and Gaussian random coefficients. The simulation runs to $t=30$, from a random initial state toward a coarsened phase pattern. We use the following first-order operator splitting scheme.

\smallskip
\noindent \textbf{Step 1}: In $[t, t + \Delta t]$, solve the nonlinear reaction part,
\begin{equation}
    \begin{aligned}
        \V{u}_1 &\longleftarrow \V{u}^t, \\
        \V{u}_2 &= \frac{\V{u}_1}{\sqrt{e^{-2 \Delta t} + (1 - e^{-2 \Delta t})(\V{u}_1)^2}}.
    \end{aligned}
\end{equation}
\noindent \textbf{Step 2}: In $[t, t + \Delta t]$, solve the following linear equation with initial state $\V{u}_2$,
\begin{equation}
    \begin{aligned}
        \frac{d \V{u}_3}{dt} &= \epsilon^2 \Delta \V{u}_3, \\
        \V{u}^{t+1} &\longleftarrow \V{u}_3.
    \end{aligned}
\end{equation}

The reduced-order approximation is applied to Step 2:
\begin{equation}
    \av{\tilde{\M \Phi}, \frac{\M \Phi \V a_3 - \M \Phi \V a_2}{\Delta t}}_{\Set{V}} = \av{\tilde{\M \Phi}, \epsilon^2 \Delta \M \Phi \V a_3}_{\Set{V}}.
\end{equation}
\begin{figure}[ht]
    \centering
    \includegraphics[width=1\textwidth]{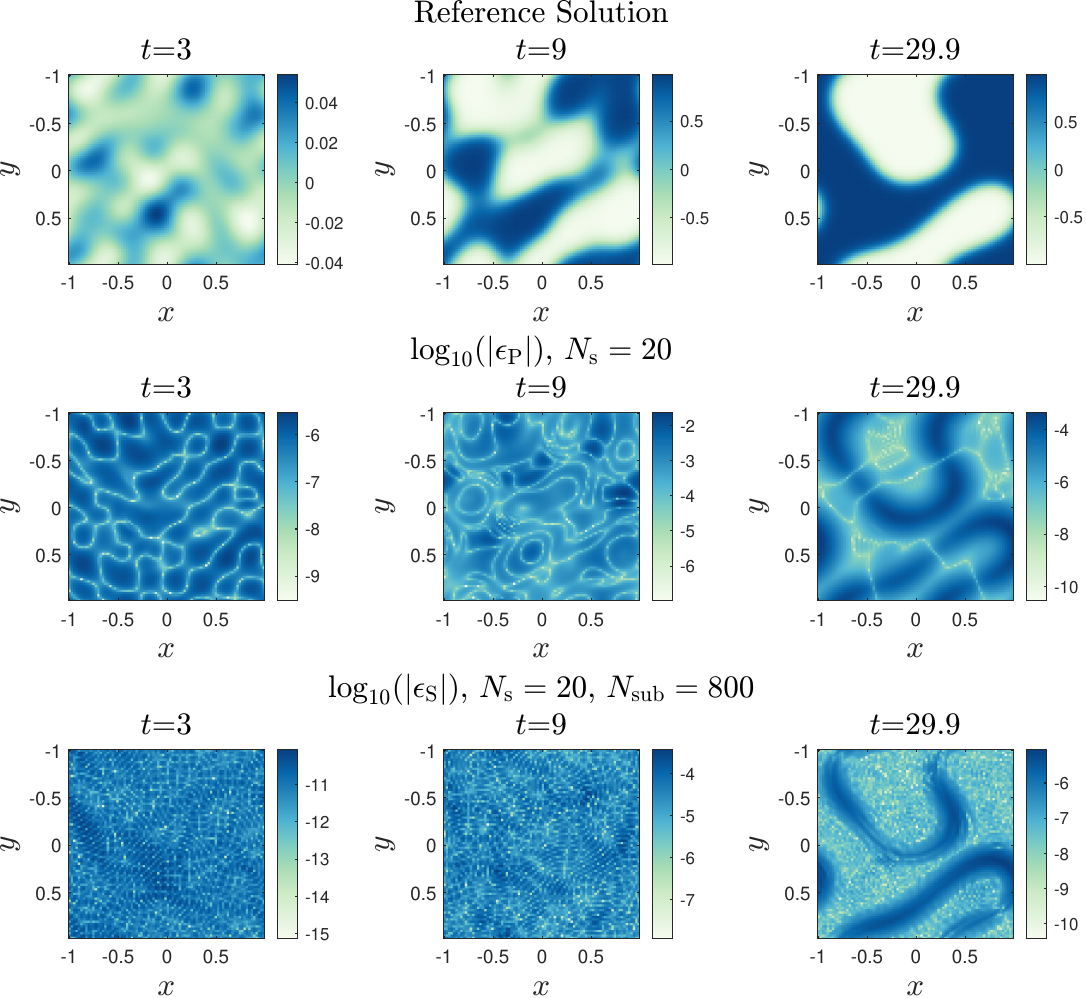}
    \caption{Interfacial coarsening in the Allen-Cahn equation. For the reported configuration, SS-POD ($N_{\mathrm{basis}}=1102, \epsilon_\mathrm{S}^\mathrm{r}=1.96\times 10^{-5}$) gives a lower relative error than POD ($6.78\times 10^{-4}$).}
    \label{fig:ref_sol_AC}
\end{figure}
The error of this operator-splitting scheme is $\mathcal{O}(\Delta t)$. If needed, other schemes such as the second-order Strang splitting scheme~\cite{strang1968construction} can be used as well and combined with SS-POD. We use the first-order splitting throughout the comparison here.

Fig.~\ref{fig:ROSG_numerical_tests}(f) shows larger error during rapid coarsening ($t=9$s) than near the later state ($t=29.9$s). Fig.~\ref{fig:ref_sol_AC} shows the interface evolution and the corresponding error distribution under limited snapshots.

\subsection{Laplace-Beltrami Problem on a Sphere and Point-wise Convergence}\label{numerical:LB}
Finally, we test SS-POD on a non-Euclidean geometry by considering a parametrized surface elliptic problem on the unit sphere:
\begin{equation}\label{eqn:laplace_beltrami}
    \Delta_{3S} u(x,y,z;\alpha) = f(x,y,z;\alpha),
    \quad (x,y,z) \in \mathbb{S}^2,
\end{equation}
where $\mathbb S^2=\{(x,y,z)\in\mathbb R^3:x^2+y^2+z^2=1\}$. The source term is constructed from a localized oscillatory profile adapted from spherical benchmark functions~\cite{Flyer2007}, with an additional parameter-dependent modulation. It is given by
\begin{equation}
    f(x,y,z;\alpha) = F''(g) (|\vec{k}|^2 - g^2) - 2 F'(g) g.
\end{equation}
The surface Laplacian is written in spherical coordinates $(\theta,\phi)$ as
\begin{equation}
    \Delta_{3S} = \frac{1}{\sin \theta}\frac{\partial}{\partial \theta}\left( \sin \theta \frac{\partial}{\partial \theta}\right) + \frac{1}{\sin^2 \theta} \frac{\partial^2}{\partial \phi^2}.
\end{equation}
The auxiliary function $F(g)$ and its derivatives are
\begin{align}
    F(g) &= \cos(g) e^{-\sigma g^2}, \\
    F'(g) &= e^{-\sigma g^2}(-\sin g - 2 \sigma g \cos g), \\
    F''(g) &= e^{-\sigma g^2} [4\sigma^2 g^2 \cos g - \cos g - 2\sigma \cos g + 4\sigma g \sin g].
\end{align}
Here $g=\vec k\cdot\vec r$, with $\vec r=(x,y,z)$ and $\vec k=(3,3,\alpha)$. The parameter $\alpha\in[1,4]$ changes the orientation and local frequency of the wave packet, and the smoothing parameter is fixed at $\sigma=0.005$.
The corresponding analytical reference solution is
\begin{equation}
    u(x,y,z;\alpha) = \cos(g) \exp(-\sigma g^2).
\end{equation}
The resulting solution family contains localized gradients and parameter-sensitive oscillations. Representative solutions are shown in the first row of Fig.~\ref{fig:ref_sol_sl}.

We use spherical harmonics as the spectral prior:
\begin{equation}
    Y_{l,m}(\theta, \phi) = (-1)^{m} \sqrt{\frac{2l+1}{4\pi}\frac{(l-m)!}{(l+m)!}} P_l^m(\cos \theta) e^{im\phi},
\end{equation}
where $P_l^m$ are the associated Legendre polynomials. Because the zero mode $Y_{0,0}$ lies in the nullspace of the Laplace-Beltrami operator, the solution is determined only up to an additive constant unless a mean constraint is imposed. The reported errors are therefore computed after removing the $Y_{0,0}$ component. For SS-POD, the two-dimensional harmonic index $(l,m)$ is reordered into a one-dimensional sequence by increasing degree $l$.
\begin{figure}[ht]
    \centering
    \includegraphics[width=1\textwidth]{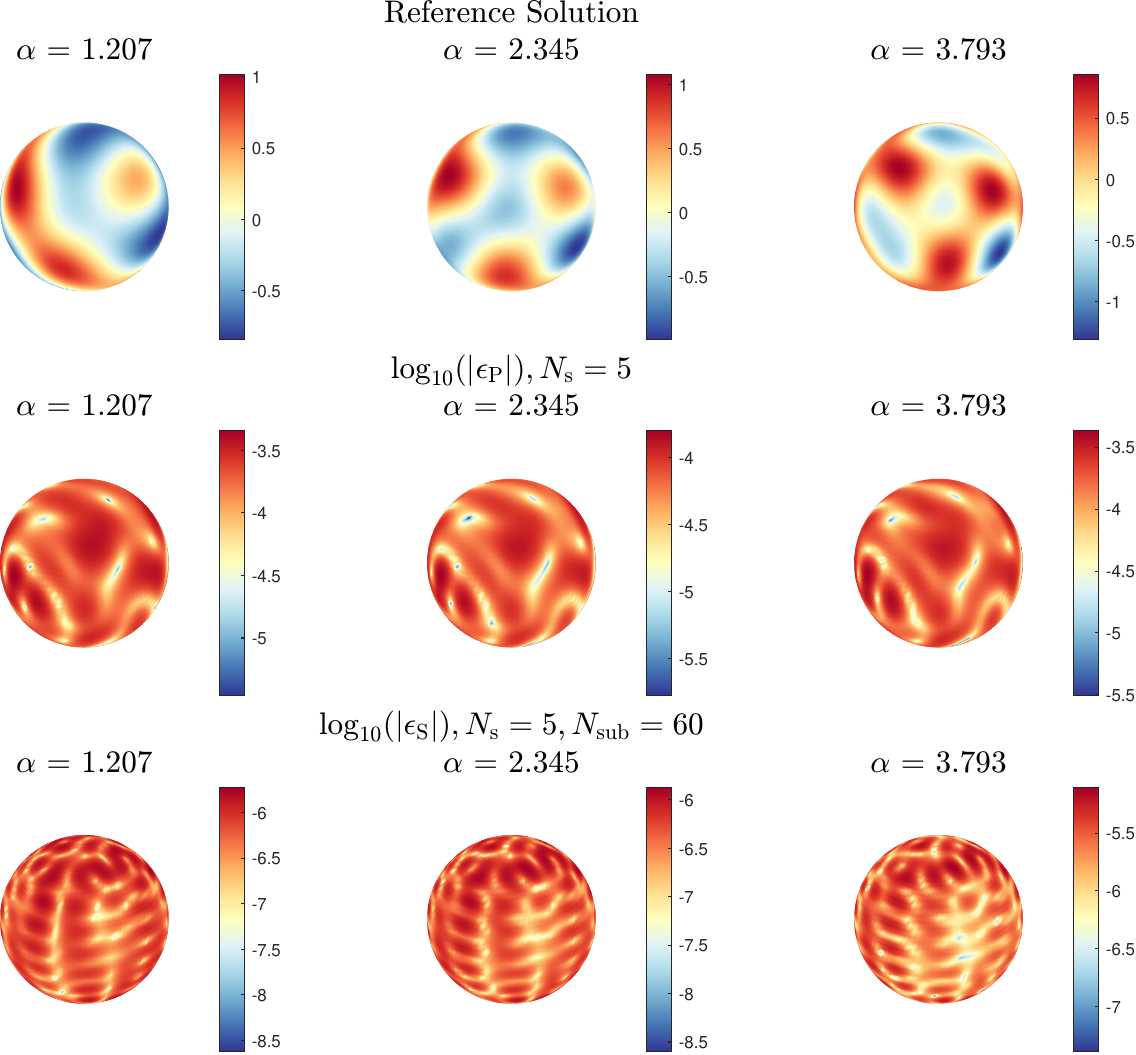}
    \caption{Laplace-Beltrami benchmarks on a unit sphere. The top row displays the reference solutions for different $\alpha$. The middle and bottom rows show the log-scale error distributions for standard POD ($N_{\mathrm{basis}}=5, \epsilon_{\mathrm{P}}^\mathrm{r}=1.85 \times 10^{-4}$) and SS-POD ($N_{\mathrm{basis}}=83, \epsilon_\mathrm{S}^\mathrm{r}=2.14 \times 10^{-6}$), respectively.}
    \label{fig:ref_sol_sl}
\end{figure}

Fig.~\ref{fig:ref_sol_sl} shows the error distribution for $N_{\mathrm{s}}=5$. Standard POD yields a relative error of $1.85 \times 10^{-4}$ with 5 basis functions, whereas SS-POD gives $2.14 \times 10^{-6}$ with 83 basis functions. The POD error concentrates near sharper solution variation; the SS-POD error is lower and more evenly distributed in this test.

Fig.~\ref{fig:comparison_POD_ROSG_sl} gives the parameter-space error distribution. The left panel reports global convergence, and the right panel shows point-wise errors over $\alpha \in [1,4]$ for $N_{\mathrm{s}}=5$. For $N_{\mathrm{sub}}=1$, the error is smallest near training snapshot locations and larger at unobserved parameter values. Increasing $N_{\mathrm{sub}}$ produces a more uniform error profile in this example.

\begin{figure}[ht]
    \centering
    \includegraphics[width=1\textwidth]{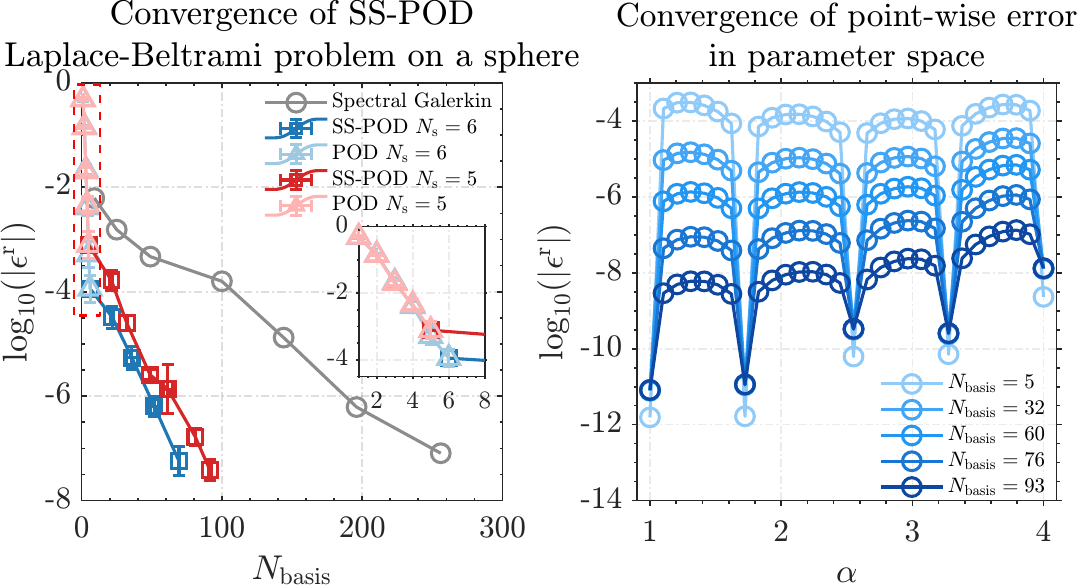}
    \caption{Generalization and convergence analysis for the Laplace-Beltrami problem. (Left) Relative error as a function of the number of basis functions $N_{\mathrm{basis}}$. (Right) Point-wise error across the parameter space $\alpha$ for different basis sizes.}
    \label{fig:comparison_POD_ROSG_sl}
\end{figure}

\section{Conclusions and Outlook}\label{sec:conclusion}

This paper introduced SS-POD, a reduced-basis construction that combines POD with a problem-adapted spectral prior. The method decomposes a finite spectral approximation space into orthogonal subspaces selected by an energy-balanced partition, projects the snapshots onto these subspaces, and performs POD locally before assembling the augmented reduced basis. The SS-POD/DEIM extension applies the same construction to nonlinear snapshots.

Across the reported benchmarks, SS-POD gives lower limited-snapshot errors than standard POD in the tested configurations and uses fewer modes than a purely spectral approximation in several cases. The examples include elliptic, parabolic, nonlinear time-dependent, and spherical problems. The common pattern is that a structured spectral prior reduces the sensitivity of POD-Galerkin ROMs to sparse or uneven snapshot information.

SS-POD remains a linear-subspace ROM, so its performance is still constrained by the intrinsic approximability of the solution manifold. For transport-dominated problems, moving discontinuities, sharp traveling layers, or other regimes with slowly decaying Kolmogorov $N$-width, a compact linear trial space may be insufficient. Such problems likely require localization, nonlinear approximation, or adaptive basis strategies in addition to SS-POD.

SS-POD also requires a useful spectral prior. Fourier, Chebyshev, and spherical harmonic bases fit tensor-product domains and the sphere, but irregular geometries may not admit a convenient global spectral basis. Domain decomposition is a plausible extension: one could partition the physical domain into subregions with local spectral representations, following spectral element methods~\cite{patera1984,canuto2007} and localized reduced basis approaches~\cite{iapichino2012}. This extension is not developed here; it would require separate analysis of interface treatment, local basis coupling, and offline cost.
Another route is to learn POD-compatible subspaces on complex geometries: recent neural subspace POD work uses DeepONet-learned POD subspaces to support Krylov solves on unstructured meshes and CAD-derived domains~\cite{levreroflorencio2026nspod}. This direction addresses a limitation that SS-POD does not resolve, namely the need for a convenient spectral prior on the physical domain.

Future work should test snapshot selection more carefully. The present experiments use random snapshot selection in several tests; systematic sampling, adaptive parameter exploration, or RBF-assisted interpolation may improve parameter-space coverage. For transport-dominated regimes, SS-POD may be combined with Principal Interval Decomposition~\cite{san2015principal} or with adaptive and hybrid ROM strategies developed to mitigate the Kolmogorov barrier in multiscale kinetic transport problems~\cite{jin2025adaptive}.

\end{document}